\documentclass[11pt]{article}
\usepackage[utf8]{inputenc}
\usepackage[T1]{fontenc}
\usepackage[english]{babel}
\usepackage{amsmath}
\usepackage{amsfonts}
\usepackage{amssymb}
\usepackage{amsthm}
\usepackage{dsfont}
\usepackage{listings}
\usepackage{xcolor}
\usepackage{cite}
\usepackage{xypic}
\usepackage{pgf,tikz,pgfplots}
\usepackage{subcaption,ragged2e}
\usepackage{mathrsfs}
\usetikzlibrary{arrows}
\usepackage{verbatim}
\usepackage{mathtools}
\usepackage{multicol} 
\usepackage{hyperref}
\usepackage{comment}

\usepackage{changes}

\usepackage{pdflscape}
\usepackage{afterpage}
\usepackage[left=1in, right=1in, bottom = 1in, top = 1in]{geometry}

\newcommand*{\N}{\mathbb{N}}
\newcommand*{\Z}{\mathbb{Z}}

\newcommand*{\R}{\mathbb{R}}

\newcommand*{\asm}{\mathsf{ASM}}

\newcommand*{\order}{\mathsf{ord}}
\newcommand*{\descset}{\mathsf{Desc}}
\newcommand*{\geodsubgraph}{\Gamma}
\newcommand*{\orderset}{\mathsf{Ord}}

\renewcommand{\phi}{\varphi}
\renewcommand{\epsilon}{\varepsilon}
\renewcommand{\theta}{\vartheta}

\DeclareMathOperator{\id}{id}

\newtheorem{lem}{Lemma}[section]
\newtheorem{prop}{Proposition}[section]
\newtheorem{defn}{Definition}[section]
\newtheorem{thm}{Theorem}[section]

\newtheorem{cor}{Corollary}[section]

\newcommand*{\diag}{D}
\newcommand*{\Vic}{\mathcal{V}}

\newcommand{\Prob}{\mathbb{P}}

\title{Sandpiles on the Vicsek fractal explode with probability $\frac{1}{4}$}
\author{Nico Heizmann, Robin Kaiser, Ecaterina Sava-Huss}

\begin{document}
\maketitle
\begin{abstract}
Vicsek fractal graphs are an important class of infinite graphs with self similar properties, polynomial growth and treelike features, on which several dynamical processes such as random walks or Abelian sandpiles can be rigorously analyzed and one can  obtain explicit closed form expressions. While such processes on Vicsek fractals and on Euclidean lattices $\Z^2$ share some properties for instance in  the recurrence behaviour, many quantities related to sandpiles on Euclidean lattices are still poorly understood. The current work focuses on the stabilization and explosion of Abelian sandpiles on Vicsek fractal graphs, and we prove that a sandpile sampled from the infinite volume limit plus one additional particle stabilizes with probability $3/4$, that is, it does not stabilize almost surely and it explodes with the complementary probability $1/4$. We prove the main result by using two different approaches: one of probabilistic nature and one of algebraic flavor. The first approach is based on investigating the particles sent to the boundary of finite volumes and showing that their number stays above four with positive probability. In the second approach  we relate the question of stabilization and explosion of sandpiles in infinite volume to the order of elements of the sandpile group on finite approximations of the infinite Vicsek graph. The method applies to more general state spaces and by employing it we also find all invariant factors of the sandpile groups on the finite approximations of the infinite Vicsek fractal.
\end{abstract}

\textit{2020 Mathematics Subject Classification.} 05C81, 20K01, 60J10, 31E05.

\textit{Keywords}: Abelian sandpile, Markov chains, absorbing states, uniform spanning trees, stabilization, toppling, Vicsek fractal,  infinite volume limit, critical group, invariant factors.

\section{Introduction}\label{sec:intro}
Fractals are intriguing objects characterizing features of real systems, and they can model a broad range of phenomena from real life. Over the past decades, fractals and the behaviour of dynamical (random) processes on them have attracted both the mathematical and physical research community, and they continue to be a very active research field. An interesting research direction is to understand how the geometry of the underlying fractals influences the long term behaviour of the random processes running on them and vice versa. The class of fractals we consider here are the infinite Vicsek fractal graphs $\Vic$ and their finite approximation graphs $\Vic_n$, for $n\in\mathbb{N}$; see Figure \ref{fig:vicsek-level-0-2} for a graphical representation of the first three approximations. On such state spaces we consider a model of redistribution of chips among the vertices. This model is called \emph{the Abelian sandpile model} and was introduced on lattices by Bak, Tang and Wiesenfeld in \cite{bak_self-organized_1988} as a toy model in an attempt to explain the physical mechanism of a system that drives itself into a critical state. In \cite{bak_self-organized_1988} the authors pointed out several key properties (such as power laws and $1/f$ noise) of the model that are also commonly observed in natural phenomena. The model was later generalized to arbitrary finite graphs by Dhar \cite{dhar_self-organized_1990}.

For a finite connected graph $G=(V\cup\{s \}, E)$ with a special vertex $s$ called the sink, we assign to each vertex $v\in V$ a natural number representing the height or the mass at $x$. The Abelian sandpile model (shortly $\asm$) is defined as follows: at each discrete time step we choose a vertex $x\in V$ uniformly at random and add mass $1$ to it. If the resulting mass at $x$ is at least the degree $\deg(x)$ of $x$, then we topple $x$ by sending one unit of mass to each neighbour of $x$, and the mass reaching the sink leaves the system forever. We continue the topplings until all vertices have less mass than the number of neighbours, that is, until we reach a \emph{stable configuration}. The process of adding mass at random and performing topplings until all vertices are stable is a Markov chain on the set of stable sandpile configurations, whose recurrent states form a group called the \emph{sandpile group} or the \emph{critical group}. The unique stationary distribution for this Markov chain is the uniform measure on the recurrent configurations; for more details see  \cite{dhar_self-organized_1990, Holroyd_2008} or \cite{jarai_sandpile_2018} for a beautiful survey with many open questions and conjectures. 

During the current work, we focus on the following questions. If $G=(V,E)$ is now an infinite graph, and $G_1\subset G_2\subset \cdots \subset G$ is a sequence of finite subgraphs of $G$ such that $\cup_{n\in \mathbb N}G_n=G$, one can show  under additional assumptions on the uniform spanning forest of $G$, that the sequence $(\mu_n)_{n\in\mathbb{N}}$ of stationary distributions for the sandpile Markov chain on $G_n$ converges weakly to a limit measure $\mu$ supported on stable sandpile configurations over $G$. We call $\mu$ the \emph{sandpile infinite volume limit} of $G$, shortly $\mathsf{IVL}$. Many interesting highly non-trivial questions can be asked about sandpile statistics under the $\mathsf{IVL}$ measure. Among them, one is to compute the height probabilities for different vertices under the uniform volume limit measure; this has been done on $\Z^2$ \cite{priezzhev_structure_1994}, on the Sierpi\'nski gasket graph \cite{heizmann_height_2023} and on regular trees \cite{height_prob_trees-dhar}. To the best of our knowledge, these are the only three state spaces where the height probabilities have been addressed so far.

Another question is that of \emph{stabilization} and \emph{explosion} of a sandpile. Given a stable sandpile $\eta$ sampled from the $\mathsf{IVL}$ measure $\mu$ on the infinite graph $G$, by adding mass $1$ to a pre-defined vertex $o$, does the sandpile $\eta+\delta_o$ stabilize almost surely? If yes, then does the size of the set of toppled vertices obey a power law distribution? There is so far no general procedure to answer this question, however certain graphs have been meanwhile investigated  and partial results are available.  The sequence of consecutive topplings until stabilization is called \emph{avalanche}. In \cite{bhupatiraju_inequalities_2017}, a lower bound on $\Z^d$ in the case $d\geq 2$ and an upper bound in the case $d\geq3$ for the distribution of avalanche sizes is derived. Furthermore, using heuristic arguments a prediction about the exponent in the power law distribution of avalanche sizes on $\Z^2$ has been made in \cite{manna_large-scale_1990} and on the Sierpi\'nski gasket in \cite{daerden_sandpiles_1998}, both supported by numerical simulations. However, even the question of almost sure stabilization on recurrent graphs remains still open. In the underlying work, we answer this question for the class of infinite Vicsek fractal graphs. In particular, 
we show that the Vicsek graph is an example of a recurrent graph on which we observe non-stabilization  with probability $1/4$. Before stating the main results, we comment briefly on the proofs. We consider two different approaches in order to prove the main result. 
\begin{enumerate}
\item The first method of probabilistic nature is based on understanding the absorbing states of an additional Markov chain that we call \emph{the nested volume Markov chain}. We sample a stable sandpile from the $\mathsf{IVL}$ measure on $\Vic$ and we add a particle to the origin $o$. We then stabilize the sandpile in increasing sets $\Vic_n$ around the origin and define a process $(X_n)_{n\in\N}$, where $X_n$ denotes the number of particles collected at the boundary of $\Vic_n$  after stabilization. Stabilization in nested volumes is allowed in view of the Abelian property of the model. Using the cut point structure of the Vicsek fractal $\Vic$, we show that $(X_n)_{n\in\N}$ is indeed a Markov chain, and we compute the probability of stabilization/explosion by relating it with the probability that $(X_n)_{n\in\N}$ reaches the absorbing state $0$.

\item The second method is of algebraic nature and applies to general infinite graphs $G=(V,E)$, albeit only under rather strong conditions on the sandpile groups of the finite exhaustions of $G$. It relates stabilization in infinite volume limit to the sandpile groups of finite approximations $G_n$ of $G$. We show that a sandpile sampled from the $\mathsf{IVL}$ measure plus one particle does not stabilize almost surely on $G$ if the order of all elements in the sandpile group of $G_n$ is uniformly bounded, for every $n\in \N$. The result on the Vicsek graph then follows after analyzing its sandpile group.
\end{enumerate}
Summarizing the two approaches, we prove the following.
\begin{thm}\label{thm:vicsek-explosion}
Let $\Vic=(V,E)$ be the infinite Vicsek graph and $\mu$ be its infinite volume limit measure. Then we have
\begin{align*}
\mu(\eta+\delta_o\text{ stabilizes})=\frac34,
\end{align*}
 where $\eta:V\rightarrow\N$ is a sandpile sampled from $\mu$, and $\delta_o:V\rightarrow\N$ is the Dirac measure taking values $1$ at $o=(0,0)$ and $0$ everywhere else.
\end{thm}
The second method  yields the following relation between stabilization in infinite volume and the order of the elements of the sandpile group on finite approximations $G_n$ of the infinite graph $G$. We denote the set of all toppled vertices during the stabilization of $\eta$ by $\mathsf{T}(\eta)$ and the diameter of a subgraph $A\subseteq G$ is defined as
\begin{align*}
    \mathsf{diam}(A)=\sup\{d(v,w):\ v,w\in A\},
\end{align*}
where $d$ is the graph distance in $G$.
\begin{thm}\label{thm:relation-order-stabilziation}
Let $G$ be an infinite, connected and locally finite graph with exhaustion $(G_n)_{n\in\N}$. Assume that the infinite volume limit $\mu$ exists as a weak limit of the sequence $(\mu_n)_{n\in\N}$ of measures, where $\mu_n$ is the uniform measure on the sandpile group $\mathcal{R}_n$ of $G_n$, for every $n\in\mathbb{N}$.
 If there exists $M\in\N$ such that for all $n\in\N$ and all  $\eta_n\in\mathcal{R}_n$ we have $\order([\eta_n])<M$, then for any $x\in G$ and $R>0$ it holds
    \begin{align*}
        \mu(\mathsf{diam}(\mathsf{T}(\eta + \delta_x))>R)>c,
    \end{align*}
where $c>0$ is a constant and $\eta$ is sampled from $\mu$.
\end{thm}
While the second method is more general and applies to a big class of infinite graphs $G$, the first method is stronger in the sense that it leads to the exact value of the probability of stabilizing a sandpile plus one additional particle on the Vicsek graph. 

\textbf{Organization of the paper.} Section \ref{sec:prelim} defines the infinite Vicsek fractal $\Vic$, its finite approximations $\Vic_n$ for $n\in \N$, and the Abelian sandpile model together with the associated Markov chain on stable sandpile configurations.  Section \ref{sec:nested-volume-method} proves Theorem \ref{thm:vicsek-explosion} by introducing the \emph{nested volume Markov chain} and relating its absorbing states with the explosion of the Abelian sandpile. In Section \ref{sec:relation-order-stabilization} we prove Theorem \ref{thm:relation-order-stabilziation} and derive Theorem \ref{thm:vicsek-explosion} as a corollary.  Section \ref{sec:vicsek-sandpile-group} analyses the sandpile group of $\Vic_n$ and the orders of its elements, for any $n\in\N$.

\section{Preliminaries}\label{sec:prelim}

\subsection{Vicsek fractal graphs}

The $n$-level Vicsek graph, for every $n\in\mathbb{N}$, is a graph that can be illustrated by embedding it in the Euclidean plane and will throughout be denoted by $\Vic_n=(V_n,E_n)$ and defined inductively as follows. The $0$-level  Vicsek graph $\Vic_0 = (V_0,E_0)$ is defined as the complete graph with $4$ vertices parametrized by $V_0 = \{(0,0),(1,0),(0,1),(1,1)\}$ and edges among them by $E_0=\{\{v,w\}:v,w\in V_0\}$. Five shifted copies of $\Vic_{n-1}$ form the graph $\Vic_n = (V_n, E_n)$ with vertex set $V_n$ 
$$ V_n = V_{n-1} \cup \big( (3^{n-1},3^{n-1}) + V_{n-1} \big) \cup \big( (2\cdot 3^{n-1},0) + V_{n-1} \big)
    \\ \cup \big( (0,2\cdot 3^{n-1}) + V_{n-1} \big)\cup \big( 2\cdot (3^{n-1},3^{n-1}) + V_{n-1} \big)$$
and edge set $E_n$
$$E_n = E_{n-1} \cup \big( (3^{n-1},3^{n-1}) + E_{n-1} \big) \cup \big( (2\cdot 3^{n-1},0) + E_{n-1} \big)
    \\ \cup \big( (0,2\cdot 3^{n-1}) + E_{n-1} \big) \cup \big( 2\cdot (3^{n-1},3^{n-1}) + E_{n-1} \big)$$
where for $(x,y)\in \R^2$, we use the notation $V+(x,y)=\{(a+x,b+y): (a,b)\in V\}$ for any set  $V\in \R^2$ and $E+(x,y)=\{\{v+(x,y),w+(x,y)\}:\{v,w\}\in E\}$ for any set of edges $E\subset \{\{x,y\}:x,y\in V\}$. The infinite Vicsek graph $\Vic$ is then defined as the union of all finite level Vicsek graphs
\begin{align*}
    \Vic = \Big(\bigcup_{n\in\N}V_n, \bigcup_{n\in\N}E_n \Big).
\end{align*}

\begin{figure}
    \centering
    \begin{tabular}{ccc}
        \begin{tikzpicture}[scale=0.85, transform shape]
            \node[shape=circle,draw=black] (A) at (0,0) {};
            \node[shape=circle,draw=black] (B) at (1,0) {};
            \node[shape=circle,draw=black] (C) at (0,1) {};
            \node[shape=circle,draw=black] (D) at (1,1) {};
            \draw[color=black] (A) -- (B) -- (C) -- (D) -- (A) -- (C);
            \draw[color=black] (B) -- (D);
        \end{tikzpicture} &
        \begin{tikzpicture}[scale=0.85, transform shape]
            \node[shape=circle,draw=black] (A) at (0,0) {};
            \node[shape=circle,draw=black] (B) at (1,0) {};
            \node[shape=circle,draw=black] (C) at (0,1) {};
            \node[shape=circle,draw=black] (D) at (1,1) {};
            \draw[color=black] (A) -- (B) -- (C) -- (D) -- (A) -- (C);
            \draw[color=black] (B) -- (D);
            \begin{scope}[shift={(1,1)}]
                \node[shape=circle,draw=black] (A) at (0,0) {};
                \node[shape=circle,draw=black] (B) at (1,0) {};
                \node[shape=circle,draw=black] (C) at (0,1) {};
                \node[shape=circle,draw=black] (D) at (1,1) {};
                \draw[color=black] (A) -- (B) -- (C) -- (D) -- (A) -- (C);
                \draw[color=black] (B) -- (D);
            \end{scope}
            \begin{scope}[shift={(2,0)}]
                \node[shape=circle,draw=black] (A) at (0,0) {};
                \node[shape=circle,draw=black] (B) at (1,0) {};
                \node[shape=circle,draw=black] (C) at (0,1) {};
                \node[shape=circle,draw=black] (D) at (1,1) {};
                \draw[color=black] (A) -- (B) -- (C) -- (D) -- (A) -- (C);
                \draw[color=black] (B) -- (D);
            \end{scope}
            \begin{scope}[shift={(2,2)}]
                \node[shape=circle,draw=black] (A) at (0,0) {};
                \node[shape=circle,draw=black] (B) at (1,0) {};
                \node[shape=circle,draw=black] (C) at (0,1) {};
                \node[shape=circle,draw=black] (D) at (1,1) {};
                \draw[color=black] (A) -- (B) -- (C) -- (D) -- (A) -- (C);
                \draw[color=black] (B) -- (D);
            \end{scope}
            \begin{scope}[shift={(0,2)}]
                \node[shape=circle,draw=black] (A) at (0,0) {};
                \node[shape=circle,draw=black] (B) at (1,0) {};
                \node[shape=circle,draw=black] (C) at (0,1) {};
                \node[shape=circle,draw=black] (D) at (1,1) {};
                \draw[color=black] (A) -- (B) -- (C) -- (D) -- (A) -- (C);
                \draw[color=black] (B) -- (D);
            \end{scope}
        \end{tikzpicture} &
        \begin{tikzpicture}[scale=0.85, transform shape]
            \node[shape=circle,draw=black] (A) at (0,0) {};
            \node[shape=circle,draw=black] (B) at (1,0) {};
            \node[shape=circle,draw=black] (C) at (0,1) {};
            \node[shape=circle,draw=black] (D) at (1,1) {};
            \draw[color=black] (A) -- (B) -- (C) -- (D) -- (A) -- (C);
            \draw[color=black] (B) -- (D);
            \begin{scope}[shift={(1,1)}]
                \node[shape=circle,draw=black] (A) at (0,0) {};
                \node[shape=circle,draw=black] (B) at (1,0) {};
                \node[shape=circle,draw=black] (C) at (0,1) {};
                \node[shape=circle,draw=black] (D) at (1,1) {};
                \draw[color=black] (A) -- (B) -- (C) -- (D) -- (A) -- (C);
                \draw[color=black] (B) -- (D);
            \end{scope}
            \begin{scope}[shift={(2,0)}]
                \node[shape=circle,draw=black] (A) at (0,0) {};
                \node[shape=circle,draw=black] (B) at (1,0) {};
                \node[shape=circle,draw=black] (C) at (0,1) {};
                \node[shape=circle,draw=black] (D) at (1,1) {};
                \draw[color=black] (A) -- (B) -- (C) -- (D) -- (A) -- (C);
                \draw[color=black] (B) -- (D);
            \end{scope}
            \begin{scope}[shift={(2,2)}]
                \node[shape=circle,draw=black] (A) at (0,0) {};
                \node[shape=circle,draw=black] (B) at (1,0) {};
                \node[shape=circle,draw=black] (C) at (0,1) {};
                \node[shape=circle,draw=black] (D) at (1,1) {};
                \draw[color=black] (A) -- (B) -- (C) -- (D) -- (A) -- (C);
                \draw[color=black] (B) -- (D);
            \end{scope}
            \begin{scope}[shift={(0,2)}]
                \node[shape=circle,draw=black] (A) at (0,0) {};
                \node[shape=circle,draw=black] (B) at (1,0) {};
                \node[shape=circle,draw=black] (C) at (0,1) {};
                \node[shape=circle,draw=black] (D) at (1,1) {};
                \draw[color=black] (A) -- (B) -- (C) -- (D) -- (A) -- (C);
                \draw[color=black] (B) -- (D);
            \end{scope}
            \begin{scope}[shift={(3,3)}]
                \node[shape=circle,draw=black] (A) at (0,0) {};
                \node[shape=circle,draw=black] (B) at (1,0) {};
                \node[shape=circle,draw=black] (C) at (0,1) {};
                \node[shape=circle,draw=black] (D) at (1,1) {};
                \draw[color=black] (A) -- (B) -- (C) -- (D) -- (A) -- (C);
                \draw[color=black] (B) -- (D);
                \begin{scope}[shift={(1,1)}]
                    \node[shape=circle,draw=black] (A) at (0,0) {};
                    \node[shape=circle,draw=black] (B) at (1,0) {};
                    \node[shape=circle,draw=black] (C) at (0,1) {};
                    \node[shape=circle,draw=black] (D) at (1,1) {};
                    \draw[color=black] (A) -- (B) -- (C) -- (D) -- (A) -- (C);
                    \draw[color=black] (B) -- (D);
                \end{scope}
                \begin{scope}[shift={(2,0)}]
                    \node[shape=circle,draw=black] (A) at (0,0) {};
                    \node[shape=circle,draw=black] (B) at (1,0) {};
                    \node[shape=circle,draw=black] (C) at (0,1) {};
                    \node[shape=circle,draw=black] (D) at (1,1) {};
                    \draw[color=black] (A) -- (B) -- (C) -- (D) -- (A) -- (C);
                    \draw[color=black] (B) -- (D);
                \end{scope}
                \begin{scope}[shift={(2,2)}]
                    \node[shape=circle,draw=black] (A) at (0,0) {};
                    \node[shape=circle,draw=black] (B) at (1,0) {};
                    \node[shape=circle,draw=black] (C) at (0,1) {};
                    \node[shape=circle,draw=black] (D) at (1,1) {};
                    \draw[color=black] (A) -- (B) -- (C) -- (D) -- (A) -- (C);
                    \draw[color=black] (B) -- (D);
                \end{scope}
                \begin{scope}[shift={(0,2)}]
                    \node[shape=circle,draw=black] (A) at (0,0) {};
                    \node[shape=circle,draw=black] (B) at (1,0) {};
                    \node[shape=circle,draw=black] (C) at (0,1) {};
                    \node[shape=circle,draw=black] (D) at (1,1) {};
                    \draw[color=black] (A) -- (B) -- (C) -- (D) -- (A) -- (C);
                    \draw[color=black] (B) -- (D);
                \end{scope}
            \end{scope}
            \begin{scope}[shift={(6,0)}]
                \node[shape=circle,draw=black] (A) at (0,0) {};
                \node[shape=circle,draw=black] (B) at (1,0) {};
                \node[shape=circle,draw=black] (C) at (0,1) {};
                \node[shape=circle,draw=black] (D) at (1,1) {};
                \draw[color=black] (A) -- (B) -- (C) -- (D) -- (A) -- (C);
                \draw[color=black] (B) -- (D);
                \begin{scope}[shift={(1,1)}]
                    \node[shape=circle,draw=black] (A) at (0,0) {};
                    \node[shape=circle,draw=black] (B) at (1,0) {};
                    \node[shape=circle,draw=black] (C) at (0,1) {};
                    \node[shape=circle,draw=black] (D) at (1,1) {};
                    \draw[color=black] (A) -- (B) -- (C) -- (D) -- (A) -- (C);
                    \draw[color=black] (B) -- (D);
                \end{scope}
                \begin{scope}[shift={(2,0)}]
                    \node[shape=circle,draw=black] (A) at (0,0) {};
                    \node[shape=circle,draw=black] (B) at (1,0) {};
                    \node[shape=circle,draw=black] (C) at (0,1) {};
                    \node[shape=circle,draw=black] (D) at (1,1) {};
                    \draw[color=black] (A) -- (B) -- (C) -- (D) -- (A) -- (C);
                    \draw[color=black] (B) -- (D);
                \end{scope}
                \begin{scope}[shift={(2,2)}]
                    \node[shape=circle,draw=black] (A) at (0,0) {};
                    \node[shape=circle,draw=black] (B) at (1,0) {};
                    \node[shape=circle,draw=black] (C) at (0,1) {};
                    \node[shape=circle,draw=black] (D) at (1,1) {};
                    \draw[color=black] (A) -- (B) -- (C) -- (D) -- (A) -- (C);
                    \draw[color=black] (B) -- (D);
                \end{scope}
                \begin{scope}[shift={(0,2)}]
                    \node[shape=circle,draw=black] (A) at (0,0) {};
                    \node[shape=circle,draw=black] (B) at (1,0) {};
                    \node[shape=circle,draw=black] (C) at (0,1) {};
                    \node[shape=circle,draw=black] (D) at (1,1) {};
                    \draw[color=black] (A) -- (B) -- (C) -- (D) -- (A) -- (C);
                    \draw[color=black] (B) -- (D);
                \end{scope}
            \end{scope}
            \begin{scope}[shift={(0,6)}]
                \node[shape=circle,draw=black] (A) at (0,0) {};
                \node[shape=circle,draw=black] (B) at (1,0) {};
                \node[shape=circle,draw=black] (C) at (0,1) {};
                \node[shape=circle,draw=black] (D) at (1,1) {};
                \draw[color=black] (A) -- (B) -- (C) -- (D) -- (A) -- (C);
                \draw[color=black] (B) -- (D);
                \begin{scope}[shift={(1,1)}]
                    \node[shape=circle,draw=black] (A) at (0,0) {};
                    \node[shape=circle,draw=black] (B) at (1,0) {};
                    \node[shape=circle,draw=black] (C) at (0,1) {};
                    \node[shape=circle,draw=black] (D) at (1,1) {};
                    \draw[color=black] (A) -- (B) -- (C) -- (D) -- (A) -- (C);
                    \draw[color=black] (B) -- (D);
                \end{scope}
                \begin{scope}[shift={(2,0)}]
                    \node[shape=circle,draw=black] (A) at (0,0) {};
                    \node[shape=circle,draw=black] (B) at (1,0) {};
                    \node[shape=circle,draw=black] (C) at (0,1) {};
                    \node[shape=circle,draw=black] (D) at (1,1) {};
                    \draw[color=black] (A) -- (B) -- (C) -- (D) -- (A) -- (C);
                    \draw[color=black] (B) -- (D);
                \end{scope}
                \begin{scope}[shift={(2,2)}]
                    \node[shape=circle,draw=black] (A) at (0,0) {};
                    \node[shape=circle,draw=black] (B) at (1,0) {};
                    \node[shape=circle,draw=black] (C) at (0,1) {};
                    \node[shape=circle,draw=black] (D) at (1,1) {};
                    \draw[color=black] (A) -- (B) -- (C) -- (D) -- (A) -- (C);
                    \draw[color=black] (B) -- (D);
                \end{scope}
                \begin{scope}[shift={(0,2)}]
                    \node[shape=circle,draw=black] (A) at (0,0) {};
                    \node[shape=circle,draw=black] (B) at (1,0) {};
                    \node[shape=circle,draw=black] (C) at (0,1) {};
                    \node[shape=circle,draw=black] (D) at (1,1) {};
                    \draw[color=black] (A) -- (B) -- (C) -- (D) -- (A) -- (C);
                    \draw[color=black] (B) -- (D);
                \end{scope}
            \end{scope}
            \begin{scope}[shift={(6,6)}]
                \node[shape=circle,draw=black] (A) at (0,0) {};
                \node[shape=circle,draw=black] (B) at (1,0) {};
                \node[shape=circle,draw=black] (C) at (0,1) {};
                \node[shape=circle,draw=black] (D) at (1,1) {};
                \draw[color=black] (A) -- (B) -- (C) -- (D) -- (A) -- (C);
                \draw[color=black] (B) -- (D);
                \begin{scope}[shift={(1,1)}]
                    \node[shape=circle,draw=black] (A) at (0,0) {};
                    \node[shape=circle,draw=black] (B) at (1,0) {};
                    \node[shape=circle,draw=black] (C) at (0,1) {};
                    \node[shape=circle,draw=black] (D) at (1,1) {};
                    \draw[color=black] (A) -- (B) -- (C) -- (D) -- (A) -- (C);
                    \draw[color=black] (B) -- (D);
                \end{scope}
                \begin{scope}[shift={(2,0)}]
                    \node[shape=circle,draw=black] (A) at (0,0) {};
                    \node[shape=circle,draw=black] (B) at (1,0) {};
                    \node[shape=circle,draw=black] (C) at (0,1) {};
                    \node[shape=circle,draw=black] (D) at (1,1) {};
                    \draw[color=black] (A) -- (B) -- (C) -- (D) -- (A) -- (C);
                    \draw[color=black] (B) -- (D);
                \end{scope}
                \begin{scope}[shift={(2,2)}]
                    \node[shape=circle,draw=black] (A) at (0,0) {};
                    \node[shape=circle,draw=black] (B) at (1,0) {};
                    \node[shape=circle,draw=black] (C) at (0,1) {};
                    \node[shape=circle,draw=black] (D) at (1,1) {};
                    \draw[color=black] (A) -- (B) -- (C) -- (D) -- (A) -- (C);
                    \draw[color=black] (B) -- (D);
                \end{scope}
                \begin{scope}[shift={(0,2)}]
                    \node[shape=circle,draw=black] (A) at (0,0) {};
                    \node[shape=circle,draw=black] (B) at (1,0) {};
                    \node[shape=circle,draw=black] (C) at (0,1) {};
                    \node[shape=circle,draw=black] (D) at (1,1) {};
                    \draw[color=black] (A) -- (B) -- (C) -- (D) -- (A) -- (C);
                    \draw[color=black] (B) -- (D);
                \end{scope}
            \end{scope}
        \end{tikzpicture}\\
        $\Vic_0$ & $\Vic_1$ & $\Vic_2$
    \end{tabular}
    \caption{The first three levels of the Vicsek graph $\Vic$.}
    \label{fig:vicsek-level-0-2}
\end{figure}
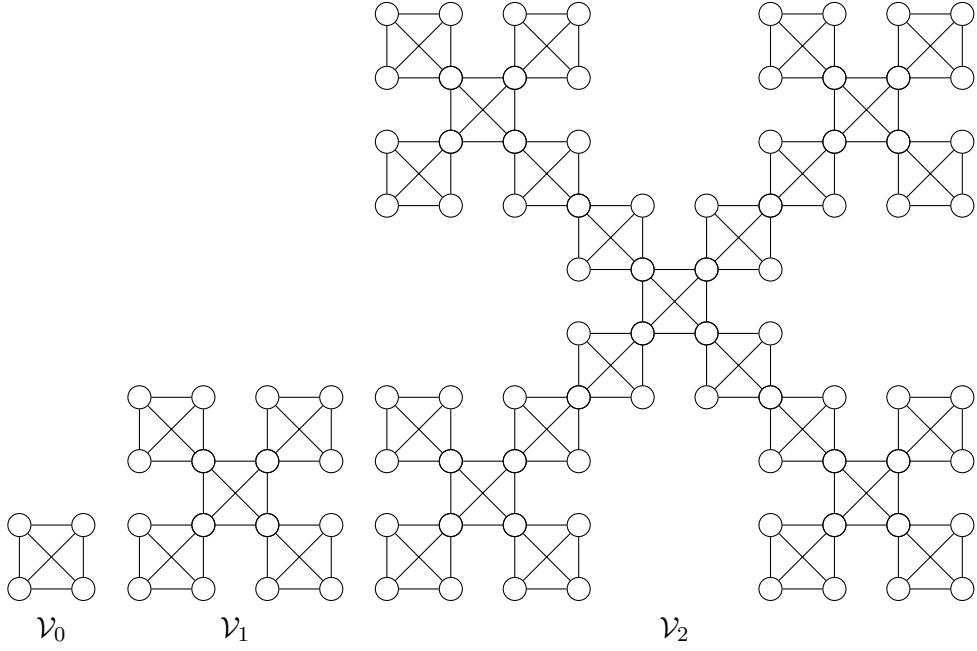

In Figure \ref{fig:vicsek-level-0-2} we illustrate the first three levels $\Vic_0,\Vic_1,\Vic_2$ of the infinite Vicsek fractal graph $\Vic$. One can define the Vicsek fractal as  the unique compact invariant set $K \subset \mathbb{R}^2$ that satisfies a self-similar property, as it is done in the fractal analysis community, see \cite{an-on-fractals}. We are interested in the discrete objects $\Vic$ and their level $n$ finite approximating graphs $\Vic_n$, for $n\in\N$.

\subsection{Abelian sandpile model}

We define the model on finite, connected graphs $G=(V\cup\{s\},E)$ with $|V|<\infty$ that have a designated vertex $s$ called the {sink}, whose role is to collect the excess mass during the stabilization of sandpiles.
If $x,y\in V\cup\{s\}$ are neighbours in $G$, i.e. $\{x,y\}\in E$, then we write $x\sim y$. For a given vertex $x\in V$ we write $\deg(x)$ for the degree of $x$ in $G$, that is, the number of neighbours of $x$ in $G$,
\begin{align*}
    \deg(x)=|\{v\in V\cup\{s\}:\ \{x,v\}\in E \}|.
\end{align*}
A \emph{sandpile configuration} or simply a \emph{sandpile} on $G$ is a function $\sigma:V\rightarrow\N$ which gives the number of particles (or chips) sitting on each vertex that is not the sink vertex. We call $\sigma$ \emph{stable} if $\sigma(v)<\deg(v)$ for all $v\in V$; otherwise $\sigma$ is called \emph{unstable}. In this case, there is a vertex $v$ with $\sigma(v)\geq\deg(v)$, that is the pile of particles at $v$ is too large and it becomes unstable and collapses. This is modeled by \emph{toppling the sandpile} $\sigma$ at $v$ with the help of the \emph{toppling operation} 
defined as
\begin{align*}
    T_v\sigma=\sigma - \Delta\delta_v,
\end{align*}
where $\delta_v:V\rightarrow\{0,1\}$ is constantly $0$ except at $v$, where it takes the value $1$, and $\Delta\in\Z^{V\times V}$ is the graph Laplacian of $G$ defined as
\begin{small}
\begin{align*}
    \Delta(x,y)=\begin{cases}
        \deg(x) &\text{if}\ x=y,\\
        -1 &\text{if}\ x\sim y,\\
        0 &\text{otherwise}.
    \end{cases}
\end{align*}\end{small}
\noindent Toppling at $v$ means that all neighbours of $v$ get one particle from $v$ and the pile at $v$ decreases its height by $\deg(v)$. The toppling is called \emph{legal}
if $\sigma(v)\geq\deg(v)$, and \emph{illegal} otherwise. Given an unstable sandpile $\sigma$, there exists a sequence of legal topplings at vertices $v_1,...,v_n$ such that $T_{v_1}...T_{v_n}\sigma$ is stable and the sequence is unique up to permutations of the vertices. The \emph{Abelian property} of the sandpile model states that no matter the order of topplings, all legal sequences end up in the same stable sandpile configuration. We define the stabilization $\sigma^{\circ}$ of $\sigma$ by
\begin{align*}
    \sigma^\circ=T_{v_1}...T_{v_n}\sigma.
\end{align*}
\textbf{Sandpile Markov chains and recurrent sandpiles.}
Based on the dynamics described so far, we introduce a discrete time Markov chain whose state space is the set of stable sandpile configurations over $G$ and it evolves in time as following. Let $Y_1,Y_2,...$ be i.i.d.\ random variables uniformly distributed over $V$ and let $\sigma_0$ be any stable sandpile on $G$. For any $n\in\N$ we define
\begin{align*}
    \sigma_{n+1}=(\sigma_n+\delta_{Y_{n+1}})^\circ.
\end{align*}
The sequence $(\sigma_n)_{n\in\N}$ is a Markov chain, where given the state $\sigma_n$ at time $n$, in order to get to the state $\sigma_{n+1}$ one picks a vertex uniformly at random in $V$, places a chip on that vertex and stabilizes the new sandpile if unstable; the stabilized configuration is then $\sigma_{n+1}$. The sequence of random stable sandpiles $(\sigma_{n})_{n\in \N}$ is called \emph{the sandpile Markov chain} and its recurrent states form an Abelian group called \emph{the sandpile group} or \emph{the critical group} of $G$, denoted by $\mathcal{R}_G$. The group operation over $\mathcal{R}_G$ is given by $\oplus$ (pointwise addition of sandpiles followed by stabilization): for sandpiles  $\eta,\zeta$, the operation $\oplus$ is defined as
\begin{align*}
    \eta\oplus \zeta=(\eta+\zeta)^\circ.
\end{align*}
Since $(\sigma_n)_{n\in\N}$ is a random walk on a finite, Abelian group, it follows from \cite{random-walk-saloffe} that the stationary distribution is the uniform distribution on $\mathcal{R}_G$. We refer the reader once again to the survey on Abelian sandpiles \cite{jarai_sandpile_2018} and for more details on the convergence to the stationary distribution of sandpile chains to \cite{abel-sand-mix-pike-levine-jerison}.

\textbf{The sandpile group.}
The sandpile group $\mathcal{R}_G$ of the graph $G$ can also be described as a factor group over $\Z^V$ in the following manner. Two configurations $\eta,\zeta\in\Z^V$ are said to be equivalent if there exists $a\in\Z^V$ such that
\begin{align*}
    \eta=\zeta+\Delta a,
\end{align*}
and we write $\eta\approx\zeta$. In words, two configurations are equivalent if a (not necessarily legal) sequence of topplings and untopplings  merging the two configurations exists. Untoppling a given vertex means that all neighbours of the given vertex send exactly one particle to that vertex. We denote the equivalence class of $\eta$ with respect to the equivalence relation ''$\approx$''  by $[\eta]$. The group $\Z^V\slash\Delta\Z^V$ is then defined as the set of equivalence classes of this relation with group operation given by
\begin{align*}
    [\eta]+[\zeta]=[\eta+\zeta].
\end{align*}
We write $\order([\eta])$ for the order of the group element $[\eta]$. It can then be shown \cite[Lemma 2.13, Lemma 2.15]{Holroyd_2008} that every equivalence class contains exactly one recurrent sandpile and it holds
\begin{align*}
    \mathcal{R}_G\cong\Z^V\slash\Delta\Z^V.
\end{align*}
\textbf{Burning bijection.}
A nice property of the recurrent sandpiles $\mathcal{R}_G$ discovered by Dhar \cite{dhar_self-organized_1990} is that they are in bijection with the spanning trees of $G$ rooted at $s$. We recall that a spanning tree $T$ of $G$ is a connected subgraph of $G$, containing all vertices of $G$ and no loops. For $v\in V$ we denote by $E_v$ the set of edges incident to $v$ and define a total ordering $<_v$ on $E_v$. Given a spanning tree $T$ of $G$ we then define
\begin{align*}
    &a_T(v)=|\{\{v,y\}\in E_v:l_T(y)<l_T(v)-1\}|,\\
    &b_T(v)=|\{\{v,y\})\in E_v:l_T(y)=l_T(v)-1\text{ and }\{v,y\}<_v e_T(v)\}|,
\end{align*}
where $l_T(v)$ is the number of edges in the unique path from $v$ to $s$ in the spanning tree $T$ and $e_T(v)$ is the first edge in this path. Then the sandpile defined by
\begin{align*}
    \sigma_T(v)=\deg(v)-1-a_T(v)-b_T(v),
\end{align*}
is recurrent and the mapping given by $T\mapsto \sigma_T$ is bijective.

\textbf{Burning algorithm.}
Dhar's burning algorithm \cite{dhar_self-organized_1990} states that on any recurrent sandpile configuration adding one particle to each vertex connected with the sink and stabilizing leads to the same configuration and each vertex topples exactly once during stabilization.

\subsection{Infinite volume limit on Vicsek graphs}

From now on, we consider Abelian sandpiles on $\Vic$ and on its level $n$ approximation graphs $\Vic_n$, $n\in\N$.
On $\Vic_n$ we set the sink vertex $s_n$ to be the upper right corner, that is $s_n=(3^n,3^n)$, and we denote by $\mathcal{R}_n$ the sandpile group of $\Vic_n$ and by $\mu_n$ the uniform distribution over $\mathcal{R}_n$.
Furthermore, we denote by $\mathsf{UST}_n$ the random variable that is uniformly distributed over the set of spanning trees of $\Vic_n$ rooted at $s_n$; we call $\mathsf{UST}_n$ \emph{the uniform spanning tree} on $\Vic_n$. 
In view of \cite{lyons-usf}, the random variables $\mathsf{UST}_n$ converge weakly to a random variable $\mathsf{UST}$ 
taking values in the set of spanning forests of $\Vic$; we call $\mathsf{UST}$ 
\emph{the uniform spanning tree} on $\Vic$. As in the case of $\Z^2$, the uniform spanning forest on $\Vic$ consists of a single component almost surely, which is the reason for which we still call $\mathsf{UST}$ uniform spanning tree on $\Vic$. Furthermore, the uniform spanning tree on the Vicsek fractal graph is one-ended almost surely, meaning that all infinite paths on it intersect infinitely often. This follows from the single cut-point structure of the Vicsek graph, where the removal of a single cut-point disconnects the graph into a finite and an infinite part. In fact, the Vicsek fractal graph itself is already one-ended.
Similarly to the case of $\Z^2$ as treated in \cite{jarai-uniform-volume-limit} (see also \cite{one-ended}), and by using the one-endedness of the 
uniform spanning tree on $\Vic$, one can show that the sequence of uniform distributions $(\mu_n)_{n\in\N}$ over $\mathcal{R}_n$ converges weakly to a measure  $\mu$ supported on stable configurations $\mathcal{R}_{\Vic}$ on $\Vic$, and this measure is called the \emph{infinite volume limit measure} (shortly $\mathsf{IVL}$) of the Abelian sandpile model. It holds $ \mathcal{R}_\Vic=\{\eta\in\Z^\Vic : \forall n\in\N:\eta|_{\Vic_n}\in\mathcal{R}_{n}\}$.
By weak convergence of the sequence $(\mu_n)_{n\in\N}$ to the measure $\mu$ on $\mathcal{R_{\Vic}}$ we mean
that for all $\varepsilon>0$, all $n\in\N$, and all $\sigma\in\mathcal{R}_{n}$, there exists $N_{\varepsilon}\in\N$
such that for all $k\geq \max\{n,N_{\varepsilon}\}$ we have
\begin{align*}
   |\mu_k(\eta|_{\Vic_n}=\sigma)-\mu(\eta|_{\Vic_n}=\sigma)|<\varepsilon,
\end{align*}
where $\eta$ is sampled according to $\mu_k$ and $\mu$ respectively.

\begin{figure}[t]
    \centering
    \begin{tikzpicture}[scale=0.75]
        \node[shape=circle,draw=orange] (A) at (0,0) {};
        \node[shape=circle,draw=orange] (B) at (1,0) {};
        \node[shape=circle,draw=orange] (C) at (0,1) {};
        \node[shape=circle,draw=orange] (D) at (1,1) {};
        \draw[color=orange] (A) -- (B) -- (C) -- (D) -- (A) -- (C);
        \draw[color=orange] (B) -- (D);
         \node (A) at (0,0) {\small\color{blue} $o$};
        \begin{scope}[shift={(1,1)}]
            \node[shape=circle,draw=orange] (A) at (0,0) {};
            \node[shape=circle,draw=orange] (B) at (1,0) {};
            \node[shape=circle,draw=orange] (C) at (0,1) {};
            \node[shape=circle,draw=orange] (D) at (1,1) {};
            \draw[color=orange] (A) -- (B) -- (C) -- (D) -- (A) -- (C);
            \draw[color=orange] (B) -- (D);
        \end{scope}
        \begin{scope}[shift={(2,2)}]
            \node[shape=circle,draw=orange] (A) at (0,0) {};
            \node[shape=circle,draw=orange] (B) at (1,0) {};
            \node[shape=circle,draw=orange] (C) at (0,1) {};
            \node[shape=circle,draw=orange] (D) at (1,1) {};
            \draw[color=orange] (A) -- (B) -- (C) -- (D) -- (A) -- (C);
            \draw[color=orange] (B) -- (D);
        \end{scope}
        \begin{scope}[shift={(2,0)}]
            \node[shape=circle,draw=black] (A) at (0,0) {};
            \node[shape=circle,draw=black] (B) at (1,0) {};
            \node[shape=circle,draw=orange] (C) at (0,1) {};
            \node[shape=circle,draw=black] (D) at (1,1) {};
            \draw (A) -- (B) -- (C) -- (D) -- (A) -- (C);
            \draw (B) -- (D);
        \end{scope}
        \begin{scope}[shift={(0,2)}]
            \node[shape=circle,draw=black] (A) at (0,0) {};
            \node[shape=circle,draw=orange] (B) at (1,0) {};
            \node[shape=circle,draw=black] (C) at (0,1) {};
            \node[shape=circle,draw=black] (D) at (1,1) {};
            \draw (A) -- (B) -- (C) -- (D) -- (A) -- (C);
            \draw (B) -- (D);
        \end{scope}
        \begin{scope}[shift={(3,3)}]
            \node[shape=circle,draw=orange] (A) at (0,0) {};
            \node[shape=circle,draw=orange] (B) at (1,0) {};
            \node[shape=circle,draw=orange] (C) at (0,1) {};
            \node[shape=circle,draw=orange] (D) at (1,1) {};
            \draw[color=orange] (A) -- (B) -- (C) -- (D) -- (A) -- (C);
            \draw[color=orange] (B) -- (D);
            \begin{scope}[shift={(1,1)}]
                \node[shape=circle,draw=orange] (A) at (0,0) {};
                \node[shape=circle,draw=orange] (B) at (1,0) {};
                \node[shape=circle,draw=orange] (C) at (0,1) {};
                \node (E) at (0,1) {\small\color{red} $x$};
                \node[shape=circle,draw=orange] (D) at (1,1) {};
                \draw[color=orange] (A) -- (B) -- (C) -- (D) -- (A) -- (C);
                \draw[color=orange] (B) -- (D);
            \end{scope}
            \begin{scope}[shift={(2,2)}]
                \node[shape=circle,draw=orange] (A) at (0,0) {};
                \node[shape=circle,draw=orange] (B) at (1,0) {};
                \node[shape=circle,draw=orange] (C) at (0,1) {};
                \node[shape=circle,draw=orange] (D) at (1,1) {};
                \draw[color=orange] (A) -- (B) -- (C) -- (D) -- (A) -- (C);
                \draw[color=orange] (B) -- (D);
            \end{scope}
            \begin{scope}[shift={(2,0)}]
                \node[shape=circle,draw=black] (A) at (0,0) {};
                \node[shape=circle,draw=black] (B) at (1,0) {};
                \node[shape=circle,draw=orange] (C) at (0,1) {};
                \node[shape=circle,draw=black] (D) at (1,1) {};
                \draw (A) -- (B) -- (C) -- (D) -- (A) -- (C);
                \draw (B) -- (D);
            \end{scope}
            \begin{scope}[shift={(0,2)}]
                \node[shape=circle,draw=red] (A) at (0,0) {};
                \node[shape=circle,draw=orange] (B) at (1,0) {};
                \node[shape=circle,draw=red] (C) at (0,1) {};
                \node[shape=circle,draw=red] (D) at (1,1) {};
                \draw[color=red] (A) -- (B) -- (C) -- (D) -- (A) -- (C);
                \draw[color=red] (B) -- (D);
            \end{scope}
        \end{scope}
        \begin{scope}[shift={(6,6)}]
            \node[shape=circle,draw=orange] (A) at (0,0) {};
            \node[shape=circle,draw=orange] (B) at (1,0) {};
            \node[shape=circle,draw=orange] (C) at (0,1) {};
            \node[shape=circle,draw=orange] (D) at (1,1) {};
            \draw[color=orange] (A) -- (B) -- (C) -- (D) -- (A) -- (C);
            \draw[color=orange] (B) -- (D);
            \begin{scope}[shift={(1,1)}]
                \node[shape=circle,draw=orange] (A) at (0,0) {};
                \node[shape=circle,draw=orange] (B) at (1,0) {};
                \node[shape=circle,draw=orange] (C) at (0,1) {};
                \node[shape=circle,draw=orange] (D) at (1,1) {};
                \draw[color=orange] (A) -- (B) -- (C) -- (D) -- (A) -- (C);
                \draw[color=orange] (B) -- (D);
            \end{scope}
            \begin{scope}[shift={(2,2)}]
                \node[shape=circle,draw=orange] (A) at (0,0) {};
                \node[shape=circle,draw=orange] (B) at (1,0) {};
                \node[shape=circle,draw=orange] (C) at (0,1) {};
                \node[shape=circle,draw=orange] (D) at (1,1) {};
                \draw[color=orange] (A) -- (B) -- (C) -- (D) -- (A) -- (C);
                \draw[color=orange] (B) -- (D);
            \end{scope}
            \begin{scope}[shift={(2,0)}]
                \node[shape=circle,draw=black] (A) at (0,0) {};
                \node[shape=circle,draw=black] (B) at (1,0) {};
                \node[shape=circle,draw=orange] (C) at (0,1) {};
                \node[shape=circle,draw=black] (D) at (1,1) {};
                \draw (A) -- (B) -- (C) -- (D) -- (A) -- (C);
                \draw (B) -- (D);
            \end{scope}
            \begin{scope}[shift={(0,2)}]
                \node[shape=circle,draw=black] (A) at (0,0) {};
                \node[shape=circle,draw=orange] (B) at (1,0) {};
                \node[shape=circle,draw=black] (C) at (0,1) {};
                \node[shape=circle,draw=black] (D) at (1,1) {};
                \draw (A) -- (B) -- (C) -- (D) -- (A) -- (C);
                \draw (B) -- (D);
            \end{scope}
        \end{scope}
        \begin{scope}[shift={(0,6)}]
            \node[shape=circle,draw=red] (A) at (0,0) {};
            \node[shape=circle,draw=red] (B) at (1,0) {};
            \node[shape=circle,draw=red] (C) at (0,1) {};
            \node[shape=circle,draw=red] (D) at (1,1) {};
            \draw[color=red] (A) -- (B) -- (C) -- (D) -- (A) -- (C);
            \draw[color=red] (B) -- (D);
            \begin{scope}[shift={(1,1)}]
                \node[shape=circle,draw=red] (A) at (0,0) {};
                \node[shape=circle,draw=red] (B) at (1,0) {};
                \node[shape=circle,draw=red] (C) at (0,1) {};
                \node[shape=circle,draw=red] (D) at (1,1) {};
                \draw[color=red] (A) -- (B) -- (C) -- (D) -- (A) -- (C);
                \draw[color=red] (B) -- (D);
            \end{scope}
            \begin{scope}[shift={(2,2)}]
                \node[shape=circle,draw=red] (A) at (0,0) {};
                \node[shape=circle,draw=red] (B) at (1,0) {};
                \node[shape=circle,draw=red] (C) at (0,1) {};
                \node[shape=circle,draw=red] (D) at (1,1) {};
                \draw[color=red] (A) -- (B) -- (C) -- (D) -- (A) -- (C);
                \draw[color=red] (B) -- (D);
            \end{scope}
            \begin{scope}[shift={(2,0)}]
                \node[shape=circle,draw=red] (A) at (0,0) {};
                \node[shape=circle,draw=red] (B) at (1,0) {};
                \node[shape=circle,draw=red] (C) at (0,1) {};
                \node[shape=circle,draw=red] (D) at (1,1) {};
                \draw[color=red] (A) -- (B) -- (C) -- (D) -- (A) -- (C);
                \draw[color=red] (B) -- (D);
            \end{scope}
            \begin{scope}[shift={(0,2)}]
                \node[shape=circle,draw=red] (A) at (0,0) {};
                \node[shape=circle,draw=red] (B) at (1,0) {};
                \node[shape=circle,draw=red] (C) at (0,1) {};
                \node[shape=circle,draw=red] (D) at (1,1) {};
                \draw[color=red] (A) -- (B) -- (C) -- (D) -- (A) -- (C);
                \draw[color=red] (B) -- (D);
            \end{scope}
        \end{scope}
        \begin{scope}[shift={(6,0)}]
            \node[shape=circle,draw=black] (A) at (0,0) {};
            \node[shape=circle,draw=black] (B) at (1,0) {};
            \node[shape=circle,draw=black] (C) at (0,1) {};
            \node[shape=circle,draw=black] (D) at (1,1) {};
            \draw (A) -- (B) -- (C) -- (D) -- (A) -- (C);
            \draw (B) -- (D);
            \begin{scope}[shift={(1,1)}]
                \node[shape=circle,draw=black] (A) at (0,0) {};
                \node[shape=circle,draw=black] (B) at (1,0) {};
                \node[shape=circle,draw=black] (C) at (0,1) {};
                \node[shape=circle,draw=black] (D) at (1,1) {};
                \draw (A) -- (B) -- (C) -- (D) -- (A) -- (C);
                \draw (B) -- (D);
            \end{scope}
            \begin{scope}[shift={(2,2)}]
                \node[shape=circle,draw=black] (A) at (0,0) {};
                \node[shape=circle,draw=black] (B) at (1,0) {};
                \node[shape=circle,draw=black] (C) at (0,1) {};
                \node[shape=circle,draw=black] (D) at (1,1) {};
                \draw (A) -- (B) -- (C) -- (D) -- (A) -- (C);
                \draw (B) -- (D);
            \end{scope}
            \begin{scope}[shift={(2,0)}]
                \node[shape=circle,draw=black] (A) at (0,0) {};
                \node[shape=circle,draw=black] (B) at (1,0) {};
                \node[shape=circle,draw=black] (C) at (0,1) {};
                \node[shape=circle,draw=black] (D) at (1,1) {};
                \draw (A) -- (B) -- (C) -- (D) -- (A) -- (C);
                \draw (B) -- (D);
            \end{scope}
            \begin{scope}[shift={(0,2)}]
                \node[shape=circle,draw=black] (A) at (0,0) {};
                \node[shape=circle,draw=black] (B) at (1,0) {};
                \node[shape=circle,draw=black] (C) at (0,1) {};
                \node[shape=circle,draw=black] (D) at (1,1) {};
                \draw (A) -- (B) -- (C) -- (D) -- (A) -- (C);
                \draw (B) -- (D);
            \end{scope}
        \end{scope}
    \end{tikzpicture}
    \caption{Illustration of $\Vic_2$ with $o=(0,0)$, $\diag_2$ (orange), $x=(4,5)$ and $T^x_2$ (red).}
    \label{fig:vicsekLev2}
\end{figure}
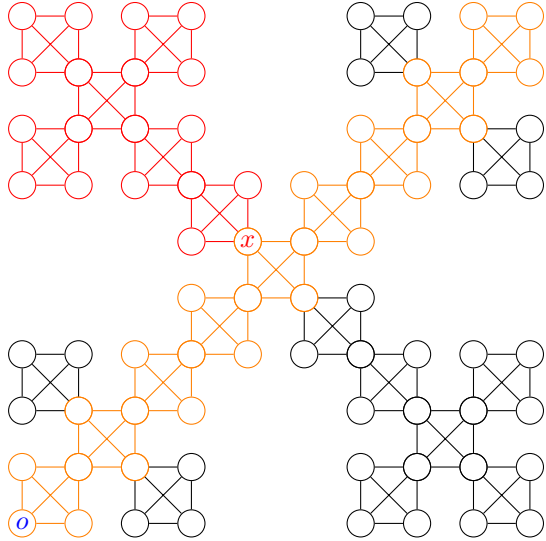

\section{Explosion of sandpiles in infinite volume}\label{sec:nested-volume-method}

This section is dedicated to the proof of Theorem \ref{thm:vicsek-explosion}.~The first step is to consider the diagonal graph $D_n$ of $\Vic_n$ as in Definition \ref{defn:diagonal-graph}, and to show that sending a particle to the sink $s_n=(3^n,3^n)$ in $D_n$ is equivalent to sending a particle to the sink in $\Vic_n$. For every $i\in\N$, we write $K^i$ for the  4-complete subgraph on the diagonal, that is, the complete graph with vertex set given by
$\{ (i-1,i-1),(i-1,i),(i,i-1),(i,i)\}$.

\begin{defn}\label{defn:diagonal-graph} For every $n\in\N$, the \emph{diagonal graph} $D_n$ of $\Vic_n$ is $\diag_n = \cup_{i=1}^{3^n} K^i$ whose vertex set $V(D_n)$ is $V(D_n)=\{(x,y)\in V_n : |x-y|\leq 1 \}$. 
Furthermore we define the diagonal vertices and the $1$-offset diagonal vertices respectively by
 \begin{align*}
  \diag^0_n &= \{(x,y)\in V_n : x=y \}, \\
  \diag^1_n &= \{(x,y)\in V_n : |x-y|= 1 \},
 \end{align*}
 thus $V(D_n)=D_n^0\cup D_n^1$. For $x\in \diag^1_n$ we define $ T^x_n$ to be the
   connected  component of $x$ away from the diagonal, that is, $T_n^x$ is the subgraph of $\Vic_n$ spanned by the vertex set
$$V(T_n^x)=\{v\in V_n : x \in t \text{ for all simple paths $t$ connecting $v$ and } \diag_n \} \backslash \{x\}.$$
\end{defn}
In other words, $T^x_n$ is the connected component of $\Vic_n\backslash \{x\}$ not containing any diagonal vertices. See Figure \ref{fig:vicsekLev2} for an illustration of $\Vic_2$, $D_2$ and an example of $T^x_2$.
We prove first that adding one chip at $o=(0,0)$ to a recurrent configuration $\eta\in\mathcal{R}_n$ on $\Vic_n$ and stabilizing, yields a recurrent configuration which equals $\eta$ on $\Vic_n\setminus D_n$.

\textit{Convention.} We will often identify a graph with its vertex set, when no confusion arises. So $v\in T_n^x$ should be understood as $v$ being a vertex in the graph $T_n^x$, or $\eta|_{T^x_n}$ should be understood as the sandpile restricted to the vertex set of $T_n^x$.

\begin{lem}\label{lem:diagonalization}
For any $n\in\N$, $\eta\in\mathcal{R}_{n}$ and all $x\in D_n^1$ we have
\begin{align*}
\big(\eta + \delta_o \big)^\circ\big|_{T^x_n} = \eta\big|_{T^x_n} .
\end{align*}
\end{lem}
\begin{proof}
Since the only connection of $T^x_n$ to the sink is through vertex $x$, $x$ acts as a sink of $T^x_n$ and any $y\in T^x_n$ can only be toppled if $x$ was toppled before. In view of Dhar's burning algorithm \cite{dhar_self-organized_1990}, for each time $x$ topples, so does every $y\in T^x_n$ and after each $y$ has toppled the configuration on $T^x_n$ is the same as before the last time $x$ was toppled.
\end{proof}
Thus sandpile configurations remain invariant on $T^x_n$, and all particles that vertex $x$ sends to $T^x_n$ will end up back in $x$ after the stabilization on the connected component $T^x_n$.
So stabilizing the sandpile $\big(\eta + \delta_o \big)$ on $\Vic_n$ with sink $s_n=(3^n,3^n)$ is the same as stabilizing $\big(\eta + \delta_o \big)|_{\diag_n}$ on the diagonal subgraph $\diag_n$. Thus, we can ignore all $T^x_n$ and analyse instead what happens during the stabilization process on $\diag_n$, because the subgraphs $T_n^x$ act only as loops where the amount of mass that goes in is returned to the off-diagonal vertex $x\in D_n^1$, and is being pushed towards the upper right corner sink $s_n$.
The next lemma shows that if we take a recurrent configuration over $\Vic_n$ and add $4$ chips to the origin $o=(0,0)$  and stabilize, the stabilization is again $\eta$. This is equivalent with $4\delta_o$ being in the equivalence class of the identity element of the sandpile group $\mathcal{R}_n$.

\begin{lem}\label{lem:4-particles-to-origin}
For any $n\in\N$ and $\eta\in\mathcal{R}_{n}$ we have
\begin{align*}
\big(\eta + 4\cdot \delta_o \big)^\circ = \eta,
\end{align*}
and all vertices in $\Vic_n$ topple at least once.
\end{lem}
\begin{figure}[hbt!]
    \centering
    \begin{subfigure}[h]{\textwidth}
        \centering
        \begin{tabular}{cccc}
              \begin{tikzpicture}
                 \node[shape=circle,draw=black] (A) at (0,0) {};
                    \node[shape=circle,draw=black] (B) at (1,0) {};
                    \node[shape=circle,draw=black] (C) at (0,1) {};
                    \node[shape=circle,draw=black] (D) at (1,1) {};
                    \draw (C) -- (A) -- (B) -- (D);
             \end{tikzpicture} & \begin{tikzpicture}
                 \node[shape=circle,draw=black] (A) at (0,0) {};
                    \node[shape=circle,draw=black] (B) at (1,0) {};
                    \node[shape=circle,draw=black] (C) at (0,1) {};
                    \node[shape=circle,draw=black] (D) at (1,1) {};
                    \draw (D) -- (C) -- (A) -- (B);
             \end{tikzpicture} & \begin{tikzpicture}
                 \node[shape=circle,draw=black] (A) at (0,0) {};
                    \node[shape=circle,draw=black] (B) at (1,0) {};
                    \node[shape=circle,draw=black] (C) at (0,1) {};
                    \node[shape=circle,draw=black] (D) at (1,1) {};
                    \draw (A) -- (C) -- (D) -- (B);
             \end{tikzpicture} & \begin{tikzpicture}
                 \node[shape=circle,draw=black] (A) at (0,0) {};
                    \node[shape=circle,draw=black] (B) at (1,0) {};
                    \node[shape=circle,draw=black] (C) at (0,1) {};
                    \node[shape=circle,draw=black] (D) at (1,1) {};
                    \draw (C) -- (D) -- (B) -- (A);
             \end{tikzpicture} \\
             \begin{tikzpicture}
                 \node[shape=circle,draw=black] (A) at (0,0) {};
                    \node[shape=circle,draw=black] (B) at (1,0) {};
                    \node[shape=circle,draw=black] (C) at (0,1) {};
                    \node[shape=circle,draw=black] (D) at (1,1) {};
                    \draw (C) -- (B) -- (A) -- (D);
             \end{tikzpicture} & \begin{tikzpicture}
                 \node[shape=circle,draw=black] (A) at (0,0) {};
                    \node[shape=circle,draw=black] (B) at (1,0) {};
                    \node[shape=circle,draw=black] (C) at (0,1) {};
                    \node[shape=circle,draw=black] (D) at (1,1) {};
                    \draw (C) -- (B) -- (D) -- (A);
             \end{tikzpicture} & \begin{tikzpicture}
                 \node[shape=circle,draw=black] (A) at (0,0) {};
                    \node[shape=circle,draw=black] (B) at (1,0) {};
                    \node[shape=circle,draw=black] (C) at (0,1) {};
                    \node[shape=circle,draw=black] (D) at (1,1) {};
                    \draw (A) -- (D) -- (C) -- (B);
             \end{tikzpicture} & \begin{tikzpicture}
                 \node[shape=circle,draw=black] (A) at (0,0) {};
                    \node[shape=circle,draw=black] (B) at (1,0) {};
                    \node[shape=circle,draw=black] (C) at (0,1) {};
                    \node[shape=circle,draw=black] (D) at (1,1) {};
                    \draw (B) -- (C) -- (A) -- (D);
             \end{tikzpicture} \\
             \begin{tikzpicture}
                 \node[shape=circle,draw=black] (A) at (0,0) {};
                    \node[shape=circle,draw=black] (B) at (1,0) {};
                    \node[shape=circle,draw=black] (C) at (0,1) {};
                    \node[shape=circle,draw=black] (D) at (1,1) {};
                    \draw (A) -- (C) -- (B) -- (D);
             \end{tikzpicture} & \begin{tikzpicture}
                 \node[shape=circle,draw=black] (A) at (0,0) {};
                    \node[shape=circle,draw=black] (B) at (1,0) {};
                    \node[shape=circle,draw=black] (C) at (0,1) {};
                    \node[shape=circle,draw=black] (D) at (1,1) {};
                    \draw (C) -- (A) -- (D) -- (B);
             \end{tikzpicture} & \begin{tikzpicture}
                 \node[shape=circle,draw=black] (A) at (0,0) {};
                    \node[shape=circle,draw=black] (B) at (1,0) {};
                    \node[shape=circle,draw=black] (C) at (0,1) {};
                    \node[shape=circle,draw=black] (D) at (1,1) {};
                    \draw (A) -- (B) -- (C) -- (D);
             \end{tikzpicture} & \begin{tikzpicture}
                 \node[shape=circle,draw=black] (A) at (0,0) {};
                    \node[shape=circle,draw=black] (B) at (1,0) {};
                    \node[shape=circle,draw=black] (C) at (0,1) {};
                    \node[shape=circle,draw=black] (D) at (1,1) {};
                    \draw (C) -- (D) -- (A) -- (B);
             \end{tikzpicture} \\
             \begin{tikzpicture}
                 \node[shape=circle,draw=black] (A) at (0,0) {};
                    \node[shape=circle,draw=black] (B) at (1,0) {};
                    \node[shape=circle,draw=black] (C) at (0,1) {};
                    \node[shape=circle,draw=black] (D) at (1,1) {};
                    \draw (A) -- (C) -- (D);
                    \draw (C) -- (B);
             \end{tikzpicture} & \begin{tikzpicture}
                 \node[shape=circle,draw=black] (A) at (0,0) {};
                    \node[shape=circle,draw=black] (B) at (1,0) {};
                    \node[shape=circle,draw=black] (C) at (0,1) {};
                    \node[shape=circle,draw=black] (D) at (1,1) {};
                    \draw (C) -- (D) -- (B);
                    \draw (A) -- (D);
             \end{tikzpicture} & \begin{tikzpicture}
                 \node[shape=circle,draw=black] (A) at (0,0) {};
                    \node[shape=circle,draw=black] (B) at (1,0) {};
                    \node[shape=circle,draw=black] (C) at (0,1) {};
                    \node[shape=circle,draw=black] (D) at (1,1) {};
                    \draw (A) -- (B) -- (D);
                    \draw (C) -- (B);
             \end{tikzpicture} & \begin{tikzpicture}
                 \node[shape=circle,draw=black] (A) at (0,0) {};
                    \node[shape=circle,draw=black] (B) at (1,0) {};
                    \node[shape=circle,draw=black] (C) at (0,1) {};
                    \node[shape=circle,draw=black] (D) at (1,1) {};
                    \draw (C) -- (A) -- (B);
                    \draw (D) -- (A);
             \end{tikzpicture}
        \end{tabular}
        \caption{Spanning trees}
        \label{fig:subfig:spanningtrees}
    \end{subfigure}\vspace{30pt}
    \begin{subfigure}[h]{\textwidth}
        \centering
        \begin{tabular}{cccc}
             \begin{tikzpicture}
                 \node[shape=circle,draw=black] (A) at (0,0) {$1$};
                    \node[shape=circle,draw=black] (B) at (1,0) {$2$};
                    \node[shape=circle,draw=black] (C) at (0,1) {$0$};
                    \node[shape=circle,draw=black] (D) at (1,1) {\phantom{$0$}};
                    \node (E) at (1,1) {s};
                    \draw (A) -- (B) -- (C) -- (D) -- (A) -- (C);
                    \draw (B) -- (D);
             \end{tikzpicture} & \begin{tikzpicture}
                 \node[shape=circle,draw=black] (A) at (0,0) {$1$};
                    \node[shape=circle,draw=black] (B) at (1,0) {$0$};
                    \node[shape=circle,draw=black] (C) at (0,1) {$2$};
                    \node[shape=circle,draw=black] (D) at (1,1) {\phantom{$0$}};
                    \node (E) at (1,1) {s};
                    \draw (A) -- (B) -- (C) -- (D) -- (A) -- (C);
                    \draw (B) -- (D);
             \end{tikzpicture} & \begin{tikzpicture}
                 \node[shape=circle,draw=black] (A) at (0,0) {$0$};
                    \node[shape=circle,draw=black] (B) at (1,0) {$2$};
                    \node[shape=circle,draw=black] (C) at (0,1) {$2$};
                    \node[shape=circle,draw=black] (D) at (1,1) {\phantom{$0$}};
                    \node (E) at (1,1) {s};
                    \draw (A) -- (B) -- (C) -- (D) -- (A) -- (C);
                    \draw (B) -- (D);
             \end{tikzpicture} & \begin{tikzpicture}
                 \node[shape=circle,draw=black] (A) at (0,0) {$1$};
                    \node[shape=circle,draw=black] (B) at (1,0) {$2$};
                    \node[shape=circle,draw=black] (C) at (0,1) {$2$};
                    \node[shape=circle,draw=black] (D) at (1,1) {\phantom{$0$}};
                    \node (E) at (1,1) {s};
                    \draw (A) -- (B) -- (C) -- (D) -- (A) -- (C);
                    \draw (B) -- (D);
             \end{tikzpicture} \\
             \begin{tikzpicture}
                 \node[shape=circle,draw=black] (A) at (0,0) {$2$};
                    \node[shape=circle,draw=black] (B) at (1,0) {$1$};
                    \node[shape=circle,draw=black] (C) at (0,1) {$0$};
                    \node[shape=circle,draw=black] (D) at (1,1) {\phantom{$0$}};
                    \node (E) at (1,1) {s};
                    \draw (A) -- (B) -- (C) -- (D) -- (A) -- (C);
                    \draw (B) -- (D);
             \end{tikzpicture} & \begin{tikzpicture}
                 \node[shape=circle,draw=black] (A) at (0,0) {$2$};
                    \node[shape=circle,draw=black] (B) at (1,0) {$2$};
                    \node[shape=circle,draw=black] (C) at (0,1) {$1$};
                    \node[shape=circle,draw=black] (D) at (1,1) {\phantom{$0$}};
                    \node (E) at (1,1) {s};
                    \draw (A) -- (B) -- (C) -- (D) -- (A) -- (C);
                    \draw (B) -- (D);
             \end{tikzpicture} & \begin{tikzpicture}
                 \node[shape=circle,draw=black] (A) at (0,0) {$2$};
                    \node[shape=circle,draw=black] (B) at (1,0) {$0$};
                    \node[shape=circle,draw=black] (C) at (0,1) {$2$};
                    \node[shape=circle,draw=black] (D) at (1,1) {\phantom{$0$}};
                    \node (E) at (1,1) {s};
                    \draw (A) -- (B) -- (C) -- (D) -- (A) -- (C);
                    \draw (B) -- (D);
             \end{tikzpicture} & \begin{tikzpicture}
                 \node[shape=circle,draw=black] (A) at (0,0) {$2$};
                    \node[shape=circle,draw=black] (B) at (1,0) {$0$};
                    \node[shape=circle,draw=black] (C) at (0,1) {$1$};
                    \node[shape=circle,draw=black] (D) at (1,1) {\phantom{$0$}};
                    \node (E) at (1,1) {s};
                    \draw (A) -- (B) -- (C) -- (D) -- (A) -- (C);
                    \draw (B) -- (D);
             \end{tikzpicture} \\
             \begin{tikzpicture}
                 \node[shape=circle,draw=black] (A) at (0,0) {$0$};
                    \node[shape=circle,draw=black] (B) at (1,0) {$2$};
                    \node[shape=circle,draw=black] (C) at (0,1) {$1$};
                    \node[shape=circle,draw=black] (D) at (1,1) {\phantom{$0$}};
                    \node (E) at (1,1) {s};
                    \draw (A) -- (B) -- (C) -- (D) -- (A) -- (C);
                    \draw (B) -- (D);
             \end{tikzpicture} & \begin{tikzpicture}
                 \node[shape=circle,draw=black] (A) at (0,0) {$2$};
                    \node[shape=circle,draw=black] (B) at (1,0) {$2$};
                    \node[shape=circle,draw=black] (C) at (0,1) {$0$};
                    \node[shape=circle,draw=black] (D) at (1,1) {\phantom{$0$}};
                    \node (E) at (1,1) {s};
                    \draw (A) -- (B) -- (C) -- (D) -- (A) -- (C);
                    \draw (B) -- (D);
             \end{tikzpicture} & \begin{tikzpicture}
                 \node[shape=circle,draw=black] (A) at (0,0) {$0$};
                    \node[shape=circle,draw=black] (B) at (1,0) {$1$};
                    \node[shape=circle,draw=black] (C) at (0,1) {$2$};
                    \node[shape=circle,draw=black] (D) at (1,1) {\phantom{$0$}};
                    \node (E) at (1,1) {s};
                    \draw (A) -- (B) -- (C) -- (D) -- (A) -- (C);
                    \draw (B) -- (D);
             \end{tikzpicture} & \begin{tikzpicture}
                 \node[shape=circle,draw=black] (A) at (0,0) {$2$};
                    \node[shape=circle,draw=black] (B) at (1,0) {$1$};
                    \node[shape=circle,draw=black] (C) at (0,1) {$2$};
                    \node[shape=circle,draw=black] (D) at (1,1) {\phantom{$0$}};
                    \node (E) at (1,1) {s};
                    \draw (A) -- (B) -- (C) -- (D) -- (A) -- (C);
                    \draw (B) -- (D);
             \end{tikzpicture} \\
             \begin{tikzpicture}
                 \node[shape=circle,draw=black] (A) at (0,0) {$1$};
                    \node[shape=circle,draw=black] (B) at (1,0) {$1$};
                    \node[shape=circle,draw=black] (C) at (0,1) {$2$};
                    \node[shape=circle,draw=black] (D) at (1,1) {\phantom{$0$}};
                    \node (E) at (1,1) {s};
                    \draw (A) -- (B) -- (C) -- (D) -- (A) -- (C);
                    \draw (B) -- (D);
             \end{tikzpicture} & \begin{tikzpicture}
                 \node[shape=circle,draw=black] (A) at (0,0) {$2$};
                    \node[shape=circle,draw=black] (B) at (1,0) {$2$};
                    \node[shape=circle,draw=black] (C) at (0,1) {$2$};
                    \node[shape=circle,draw=black] (D) at (1,1) {\phantom{$0$}};
                    \node (E) at (1,1) {s};
                    \draw (A) -- (B) -- (C) -- (D) -- (A) -- (C);
                    \draw (B) -- (D);
             \end{tikzpicture} & \begin{tikzpicture}
                 \node[shape=circle,draw=black] (A) at (0,0) {$1$};
                    \node[shape=circle,draw=black] (B) at (1,0) {$2$};
                    \node[shape=circle,draw=black] (C) at (0,1) {$1$};
                    \node[shape=circle,draw=black] (D) at (1,1) {\phantom{$0$}};
                    \node (E) at (1,1) {s};
                    \draw (A) -- (B) -- (C) -- (D) -- (A) -- (C);
                    \draw (B) -- (D);
             \end{tikzpicture} & \begin{tikzpicture}
                 \node[shape=circle,draw=black] (A) at (0,0) {$2$};
                    \node[shape=circle,draw=black] (B) at (1,0) {$1$};
                    \node[shape=circle,draw=black] (C) at (0,1) {$1$};
                    \node[shape=circle,draw=black] (D) at (1,1) {\phantom{$0$}};
                    \node (E) at (1,1) {s};
                    \draw (A) -- (B) -- (C) -- (D) -- (A) -- (C);
                    \draw (B) -- (D);
             \end{tikzpicture}
        \end{tabular}
        \caption{Recurrent sandpiles}
        \label{fig:y equals x}
     \end{subfigure}
    \caption{The spanning trees rooted at the sink and the corresponding recurrent sandpiles of $\Vic_0 = K_4$ with sink $s_0=(1,1)$ the upper right corner from the burning bijection.}
    \label{fig:recurrentK4}
\end{figure}
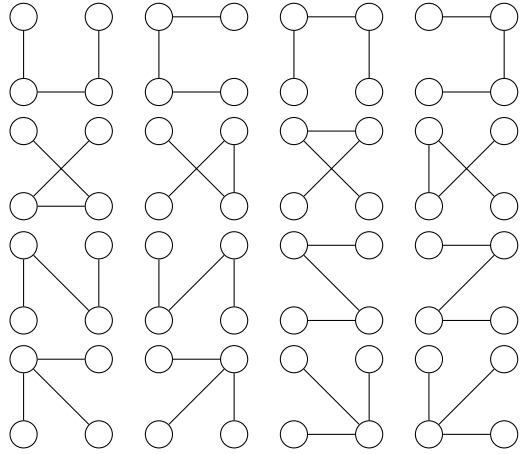
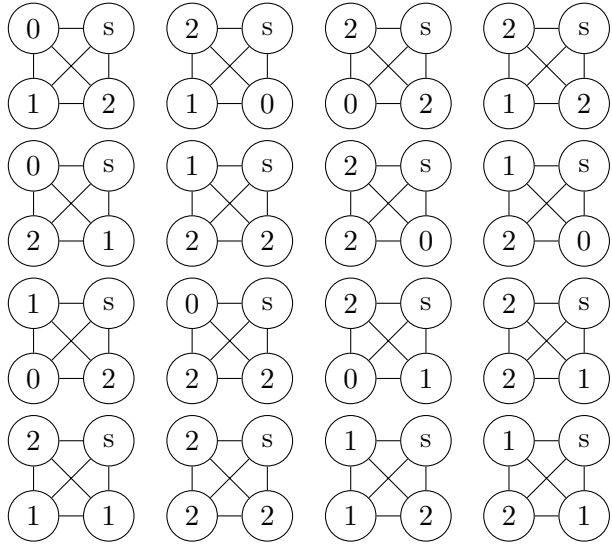
\begin{proof}
By Lemma \ref{lem:diagonalization} it is sufficient to consider only sandpiles on the diagonal $\diag_n$, and recall that if a vertex $x\in \diag^1_n$ topples, so does the whole connected component $T^x_n$. Since $\diag_n$ consists of $3^n$ copies of $K_4$ labeled by $K^{i}$, the spanning trees of $\diag_n$ consist of spanning trees of $K^{i}$ glued together, because the copies $K^{i}$ and $K^{i+1}$ share only the vertex $(i,i)$, for $i=1,\ldots,3^n$. The burning bijection implies that also the recurrent sandpiles on $\diag_n$ are recurrent sandpiles on $K^{i}$ glued together, while taking care of the common vertex $(i,i)$  by adding an additional three particles to it. We can now verify by considering all recurrent configurations on $K_4$ as seen in Figure \ref{fig:recurrentK4} that the addition of 4 particles followed by stabililzation leaves the configurations unchanged and sends 4 particles to the sink. Indeed via careful exploitation of symmetries it suffices to consider only a subcase of all configurations in Figure \ref{fig:recurrentK4}. By the burning algorithm, if the top right corner of any copy of $K_4$ is identified as the sink, the addition of 4 particles to the lower left corner of the same copy of $K_4$ and subsequent stabilisation is equivalent to the stabilisation on $K_4$ and leaves all the vertices up to this copy unchanged. The proof is now finished by carrying the 4 particles along the diagonal via subsequent stabilisation of all the copies of $K_4$ in the diagonal.
\end{proof}

It is straightforward to extend the argument from the proof of Lemma \ref{lem:4-particles-to-origin} and to place 4 particles on any diagonal vertex $x=(i,i)\in D_n^0$ and to stabilize. This process leaves again invariant the starting configuration, and the recursive proof would then start from the copy $K^i$ instead of $K^1$.
\begin{cor}\label{cor:4particlesDiag}
    For any $n\in\N$, $\eta\in\mathcal{R}_{n}$ and any $x\in D_n^0$ we have
    \begin{align*}
        \big(\eta + 4\cdot \delta_{x}\big)^\circ = \eta,
    \end{align*}
    and all vertices in $\Vic_n$ topple at least once.
\end{cor}

\subsection{Nested volume Markov chain}

We sample now a stable sandpile $\eta\in\mathcal{R}_{\Vic}$ on the infinite Vicsek fractal graph from the $\mathsf{IVL}$ measure $\mu$  and we add one particle at the origin $o=(0,0)$. We recall that our ultimate goal is to understand if the sandpile $\eta +\delta_0$ stabilizes almost surely or not, and if not, to compute the stabilization probability.
Based on the previous two Lemmas, in order to understand if stabilization occurs,
it is enough to consider the particles each 4-complete graph $K^i$ passes on to the next copy $K^{i+1}$, when stabilizing  $(\eta+\delta_o)$ successively along the diagonal copies $D_n=\cup_{i=0}^{3^n}K^i$ in ascending order starting with $K^1$. In view of Lemma \ref{lem:diagonalization}, the sandpile on the connected components $T_n^x$ rooted at the 1-offset diagonal vertices $x\in D_n^1$ stays invariant during stabilization.
Due to the Abelian property of the toppling operations, stabilizing $\eta+\delta_o$ successively along the diagonal up to $K^i$ is the same as stabilizing $\eta+\delta_o$ on $\cup_{j=1}^i K^i$. This reasoning motivates the introduction of the sequence $(X_i)_{i\in\N}$ of random variables  that count the number of particles landed in the sink $(i,i)$: for $\eta\in\mathcal{R}_{\Vic}$ sampled from $\mu$, let $X_0=1$ and for any $i\geq 1$ define
\begin{align*}
    X_i = \#\big(\text{particles sent to $(i,i)$ during the stabilization of $\eta+\delta_o$ on $\cup_{j=1}^i K^j$ with sink $(i,i)$}\big).
\end{align*}
The transition from $X_i$ to $X_{i+1}$ depends only on the values of $\eta$ in $K^{i+1}$ and the current state $X_i$. From the burning algorithm together with the fact that the uniform spanning tree on the Vicsek fractal graph consists of uniformly chosen spanning trees on all copies of $K_4$ in the Vicsek fractal graph, it follows that the restrictions $(\eta|_{K^i})_{i\in \N_0}$ are i.i.d.~samples of recurrent sandpiles on the complete graph $K_4$ as shown in Figure \ref{fig:recurrentK4} with three additional particles at all cutpoints; see also the proof of Lemma \ref{lem:4-particles-to-origin}. This makes $(X_i)_{i\in \N_0}$ a time-homogeneous Markov chain with state space $\{0,1,2,3,4\}$, for which the one-step transition probabilities can be easily computed from the transitions of the recurrent sandpiles on $K_4$ as shown in Table \ref{tab:k4-transPob}. Figure \ref{fig:k4-transitions} shows  all the possible number of particles send to the upper right corner $(i,i)$ of $K^i$ (up to permutations) if the addition of one particle at the lower right corner $(i-1,i-1)$ of $K^i$ leads to a toppling. 

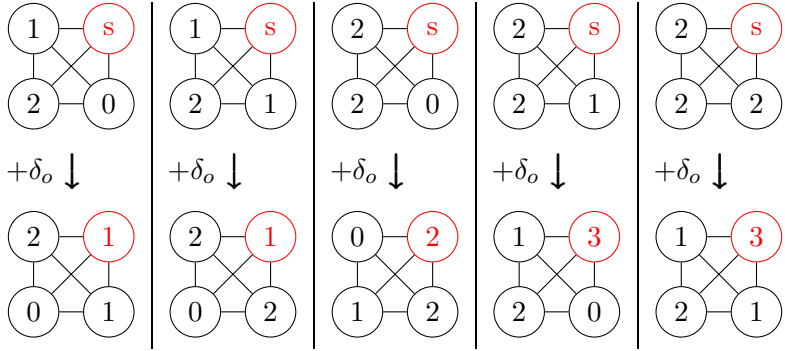
\begin{figure}[!h]
    \centering
    \begin{tabular}{c|c|c|c|c}
     \begin{tikzpicture}
                 \node[shape=circle,draw=black] (A) at (0,0) {$2$};
                    \node[shape=circle,draw=black] (B) at (1,0) {$0$};
                    \node[shape=circle,draw=black] (C) at (0,1) {$1$};
                    \node[shape=circle,draw=red] (D) at (1,1) {\phantom{$0$}};
                    \node (E) at (1,1) {\textcolor{red}{s}};
                    \draw (A) -- (B) -- (C) -- (D) -- (A) -- (C);
                    \draw (B) -- (D);
             \end{tikzpicture} & \begin{tikzpicture}
                 \node[shape=circle,draw=black] (A) at (0,0) {$2$};
                    \node[shape=circle,draw=black] (B) at (1,0) {$1$};
                    \node[shape=circle,draw=black] (C) at (0,1) {$1$};
                    \node[shape=circle,draw=red] (D) at (1,1) {\phantom{$0$}};
                    \node (E) at (1,1) {\textcolor{red}{s}};
                    \draw (A) -- (B) -- (C) -- (D) -- (A) -- (C);
                    \draw (B) -- (D);
             \end{tikzpicture} & \begin{tikzpicture}
                 \node[shape=circle,draw=black] (A) at (0,0) {$2$};
                    \node[shape=circle,draw=black] (B) at (1,0) {$0$};
                    \node[shape=circle,draw=black] (C) at (0,1) {$2$};
                    \node[shape=circle,draw=red] (D) at (1,1) {\phantom{$0$}};
                    \node (E) at (1,1) {\textcolor{red}{s}};
                    \draw (A) -- (B) -- (C) -- (D) -- (A) -- (C);
                    \draw (B) -- (D);
             \end{tikzpicture} & \begin{tikzpicture}
                 \node[shape=circle,draw=black] (A) at (0,0) {$2$};
                    \node[shape=circle,draw=black] (B) at (1,0) {$1$};
                    \node[shape=circle,draw=black] (C) at (0,1) {$2$};
                    \node[shape=circle,draw=red] (D) at (1,1) {\phantom{$0$}};
                    \node (E) at (1,1) {\textcolor{red}{s}};
                    \draw (A) -- (B) -- (C) -- (D) -- (A) -- (C);
                    \draw (B) -- (D);
             \end{tikzpicture} & \begin{tikzpicture}
                 \node[shape=circle,draw=black] (A) at (0,0) {$2$};
                    \node[shape=circle,draw=black] (B) at (1,0) {$2$};
                    \node[shape=circle,draw=black] (C) at (0,1) {$2$};
                    \node[shape=circle,draw=red] (D) at (1,1) {\phantom{$0$}};
                    \node (E) at (1,1) {\textcolor{red}{s}};
                    \draw (A) -- (B) -- (C) -- (D) -- (A) -- (C);
                    \draw (B) -- (D);
             \end{tikzpicture} \\[5pt]
             $+\delta_o\ \pmb{\big\downarrow}$ \phantom{shif}& 
             $+\delta_o\ \pmb{\big\downarrow}$ \phantom{shif}& 
             $+\delta_o\ \pmb{\big\downarrow}$ \phantom{shif}& 
             $+\delta_o\ \pmb{\big\downarrow}$ \phantom{shif}& 
             $+\delta_o\ \pmb{\big\downarrow}$ \phantom{shif} \\[8pt]
        \begin{tikzpicture}
                 \node[shape=circle,draw=black] (A) at (0,0) {$0$};
                    \node[shape=circle,draw=black] (B) at (1,0) {$1$};
                    \node[shape=circle,draw=black] (C) at (0,1) {$2$};
                    \node[shape=circle,draw=red] (D) at (1,1) {\textcolor{red}{$1$}};
                    \draw (A) -- (B) -- (C) -- (D) -- (A) -- (C);
                    \draw (B) -- (D);
             \end{tikzpicture} & \begin{tikzpicture}
                 \node[shape=circle,draw=black] (A) at (0,0) {$0$};
                    \node[shape=circle,draw=black] (B) at (1,0) {$2$};
                    \node[shape=circle,draw=black] (C) at (0,1) {$2$};
                    \node[shape=circle,draw=red] (D) at (1,1) {\textcolor{red}{$1$}};
                    \draw (A) -- (B) -- (C) -- (D) -- (A) -- (C);
                    \draw (B) -- (D);
             \end{tikzpicture} & \begin{tikzpicture}
                 \node[shape=circle,draw=black] (A) at (0,0) {$1$};
                    \node[shape=circle,draw=black] (B) at (1,0) {$2$};
                    \node[shape=circle,draw=black] (C) at (0,1) {$0$};
                    \node[shape=circle,draw=red] (D) at (1,1) {\textcolor{red}{$2$}};
                    \draw (A) -- (B) -- (C) -- (D) -- (A) -- (C);
                    \draw (B) -- (D);
             \end{tikzpicture} & \begin{tikzpicture}
                 \node[shape=circle,draw=black] (A) at (0,0) {$2$};
                    \node[shape=circle,draw=black] (B) at (1,0) {$0$};
                    \node[shape=circle,draw=black] (C) at (0,1) {$1$};
                    \node[shape=circle,draw=red] (D) at (1,1) {\textcolor{red}{$3$}};
                    \draw (A) -- (B) -- (C) -- (D) -- (A) -- (C);
                    \draw (B) -- (D);
             \end{tikzpicture} & \begin{tikzpicture}
                 \node[shape=circle,draw=black] (A) at (0,0) {$2$};
                    \node[shape=circle,draw=black] (B) at (1,0) {$1$};
                    \node[shape=circle,draw=black] (C) at (0,1) {$1$};
                    \node[shape=circle,draw=red] (D) at (1,1) {\textcolor{red}{$3$}};
                    \draw (A) -- (B) -- (C) -- (D) -- (A) -- (C);
                    \draw (B) -- (D);
             \end{tikzpicture} 
        \end{tabular}
    \caption{All recurrent sandpile configurations of $K_4$ up to permutations of vertices (top) and their transitions after adding one particle to the lower left vertex of height $2$ and the corresponding collected particles in the sink $s$ (bottom).}
    \label{fig:k4-transitions}
\end{figure}
These one-step transitions and their corresponding probabilities can be used to calculate the number of particles the copy $K^i$ passes on to $K^{i+1}$ in the course of stabilization of $\eta+\delta_o$. For any $i\in\N$, if the sink $(i,i)$ of $K^i$ does not receive any particles from $K^{i-1}$, then it cannot pass any particles to the $K^{i+1}$, so $0$ is an absorbing state of $(X_i)_{i\in\N_0}$. On the other hand, according to Corollary \ref{cor:4particlesDiag}, if $K^i$ receives $4$ particles, it will also pass on $4$ particles. Therefore $4$ is an absorbing state as well. For the cases of $1$, $2$, and $3$ particles with the help of  Figure \ref{fig:k4-transitions}, we collect the results  in Table \ref{tab:k4-transPob}.
\begin{table}[]
    \centering
    \begin{tabular}{c||c|c|c}
     configuration 
     & \begin{tikzpicture}[scale=0.6, transform shape]
                 \node[shape=circle,draw=black] (A) at (0,0) {$1$};
                    \node[shape=circle,draw=black] (B) at (1,0) {$2$};
                    \node[shape=circle,draw=black] (C) at (0,1) {$0$};
                    \node[shape=circle,draw=red] (D) at (1,1) {\phantom{$0$}};
                    \node (E) at (1,1) {\textcolor{red}{s}};
                    \draw (A) -- (B) -- (C) -- (D) -- (A) -- (C);
                    \draw (B) -- (D);
             \end{tikzpicture}, \begin{tikzpicture}[scale=0.6, transform shape]
                 \node[shape=circle,draw=black] (A) at (0,0) {$1$};
                    \node[shape=circle,draw=black] (B) at (1,0) {$0$};
                    \node[shape=circle,draw=black] (C) at (0,1) {$2$};
                    \node[shape=circle,draw=red] (D) at (1,1) {\phantom{$0$}};
                    \node (E) at (1,1) {\textcolor{red}{s}};
                    \draw (A) -- (B) -- (C) -- (D) -- (A) -- (C);
                    \draw (B) -- (D);
             \end{tikzpicture}
    & \begin{tikzpicture}[scale=0.6, transform shape]
                 \node[shape=circle,draw=black] (A) at (0,0) {$2$};
                    \node[shape=circle,draw=black] (B) at (1,0) {$1$};
                    \node[shape=circle,draw=black] (C) at (0,1) {$0$};
                    \node[shape=circle,draw=red] (D) at (1,1) {\phantom{$0$}};
                    \node (E) at (1,1) {\textcolor{red}{s}};
                    \draw (A) -- (B) -- (C) -- (D) -- (A) -- (C);
                    \draw (B) -- (D);
             \end{tikzpicture}, \begin{tikzpicture}[scale=0.6, transform shape]
                 \node[shape=circle,draw=black] (A) at (0,0) {$2$};
                    \node[shape=circle,draw=black] (B) at (1,0) {$0$};
                    \node[shape=circle,draw=black] (C) at (0,1) {$1$};
                    \node[shape=circle,draw=red] (D) at (1,1) {\phantom{$0$}};
                    \node (E) at (1,1) {\textcolor{red}{s}};
                    \draw (A) -- (B) -- (C) -- (D) -- (A) -- (C);
                    \draw (B) -- (D);
             \end{tikzpicture}, \begin{tikzpicture}[scale=0.6, transform shape]
                 \node[shape=circle,draw=black] (A) at (0,0) {$2$};
                    \node[shape=circle,draw=black] (B) at (1,0) {$1$};
                    \node[shape=circle,draw=black] (C) at (0,1) {$1$};
                    \node[shape=circle,draw=red] (D) at (1,1) {\phantom{$0$}};
                    \node (E) at (1,1) {\textcolor{red}{s}};
                    \draw (A) -- (B) -- (C) -- (D) -- (A) -- (C);
                    \draw (B) -- (D);
             \end{tikzpicture}
    & \begin{tikzpicture}[scale=0.6, transform shape]
                 \node[shape=circle,draw=black] (A) at (0,0) {$1$};
                    \node[shape=circle,draw=black] (B) at (1,0) {$1$};
                    \node[shape=circle,draw=black] (C) at (0,1) {$2$};
                    \node[shape=circle,draw=red] (D) at (1,1) {\phantom{$0$}};
                    \node (E) at (1,1) {\textcolor{red}{s}};
                    \draw (A) -- (B) -- (C) -- (D) -- (A) -- (C);
                    \draw (B) -- (D);
             \end{tikzpicture}, \begin{tikzpicture}[scale=0.6, transform shape]
                 \node[shape=circle,draw=black] (A) at (0,0) {$1$};
                    \node[shape=circle,draw=black] (B) at (1,0) {$2$};
                    \node[shape=circle,draw=black] (C) at (0,1) {$1$};
                    \node[shape=circle,draw=red] (D) at (1,1) {\phantom{$0$}};
                    \node (E) at (1,1) {\textcolor{red}{s}};
                    \draw (A) -- (B) -- (C) -- (D) -- (A) -- (C);
                    \draw (B) -- (D);
             \end{tikzpicture}, \begin{tikzpicture}[scale=0.6, transform shape]
                 \node[shape=circle,draw=black] (A) at (0,0) {$1$};
                    \node[shape=circle,draw=black] (B) at (1,0) {$2$};
                    \node[shape=circle,draw=black] (C) at (0,1) {$2$};
                    \node[shape=circle,draw=red] (D) at (1,1) {\phantom{$0$}};
                    \node (E) at (1,1) {\textcolor{red}{s}};
                    \draw (A) -- (B) -- (C) -- (D) -- (A) -- (C);
                    \draw (B) -- (D);
             \end{tikzpicture}\\ \hline
     added 
     & \begin{tabular}{c|c|c}
          1 & 2 & 3 
          \end{tabular}
    & \begin{tabular}{c|c|c}
          1 & 2 & 3 
          \end{tabular}
    & \begin{tabular}{c|c|c}
          1 & 2 & 3 
          \end{tabular}\\ \hline
     collected 
     & \begin{tabular}{c|c|c}
          0 & 2 & 2 
          \end{tabular}
    & \begin{tabular}{c|c|c}
          1 & 1 & 1 
          \end{tabular}
    & \begin{tabular}{c|c|c}
          0 & 3 & 4 
          \end{tabular} \\ \hline \hline & & & \\[-1pt]
    configuration 
    & \begin{tikzpicture}[scale=0.6, transform shape]
                 \node[shape=circle,draw=black] (A) at (0,0) {$2$};
                    \node[shape=circle,draw=black] (B) at (1,0) {$2$};
                    \node[shape=circle,draw=black] (C) at (0,1) {$0$};
                    \node[shape=circle,draw=red] (D) at (1,1) {\phantom{$0$}};
                    \node (E) at (1,1) {\textcolor{red}{s}};
                    \draw (A) -- (B) -- (C) -- (D) -- (A) -- (C);
                    \draw (B) -- (D);
             \end{tikzpicture}, \begin{tikzpicture}[scale=0.6, transform shape]
                 \node[shape=circle,draw=black] (A) at (0,0) {$2$};
                    \node[shape=circle,draw=black] (B) at (1,0) {$0$};
                    \node[shape=circle,draw=black] (C) at (0,1) {$2$};
                    \node[shape=circle,draw=red] (D) at (1,1) {\phantom{$0$}};
                    \node (E) at (1,1) {\textcolor{red}{s}};
                    \draw (A) -- (B) -- (C) -- (D) -- (A) -- (C);
                    \draw (B) -- (D);
             \end{tikzpicture}
    & \begin{tikzpicture}[scale=0.6, transform shape]
                 \node[shape=circle,draw=black] (A) at (0,0) {$0$};
                    \node[shape=circle,draw=black] (B) at (1,0) {$1$};
                    \node[shape=circle,draw=black] (C) at (0,1) {$2$};
                    \node[shape=circle,draw=red] (D) at (1,1) {\phantom{$0$}};
                    \node (E) at (1,1) {\textcolor{red}{s}};
                    \draw (A) -- (B) -- (C) -- (D) -- (A) -- (C);
                    \draw (B) -- (D);
             \end{tikzpicture}, \begin{tikzpicture}[scale=0.6, transform shape]
                 \node[shape=circle,draw=black] (A) at (0,0) {$0$};
                    \node[shape=circle,draw=black] (B) at (1,0) {$2$};
                    \node[shape=circle,draw=black] (C) at (0,1) {$1$};
                    \node[shape=circle,draw=red] (D) at (1,1) {\phantom{$0$}};
                    \node (E) at (1,1) {\textcolor{red}{s}};
                    \draw (A) -- (B) -- (C) -- (D) -- (A) -- (C);
                    \draw (B) -- (D);
             \end{tikzpicture}, \begin{tikzpicture}[scale=0.6, transform shape]
                 \node[shape=circle,draw=black] (A) at (0,0) {$0$};
                    \node[shape=circle,draw=black] (B) at (1,0) {$2$};
                    \node[shape=circle,draw=black] (C) at (0,1) {$2$};
                    \node[shape=circle,draw=red] (D) at (1,1) {\phantom{$0$}};
                    \node (E) at (1,1) {\textcolor{red}{s}};
                    \draw (A) -- (B) -- (C) -- (D) -- (A) -- (C);
                    \draw (B) -- (D);
             \end{tikzpicture}
    & \begin{tikzpicture}[scale=0.6, transform shape]
                 \node[shape=circle,draw=black] (A) at (0,0) {$2$};
                    \node[shape=circle,draw=black] (B) at (1,0) {$2$};
                    \node[shape=circle,draw=black] (C) at (0,1) {$2$};
                    \node[shape=circle,draw=red] (D) at (1,1) {\phantom{$0$}};
                    \node (E) at (1,1) {\textcolor{red}{s}};
                    \draw (A) -- (B) -- (C) -- (D) -- (A) -- (C);
                    \draw (B) -- (D);
             \end{tikzpicture}, \begin{tikzpicture}[scale=0.6, transform shape]
                 \node[shape=circle,draw=black] (A) at (0,0) {$2$};
                    \node[shape=circle,draw=black] (B) at (1,0) {$2$};
                    \node[shape=circle,draw=black] (C) at (0,1) {$1$};
                    \node[shape=circle,draw=red] (D) at (1,1) {\phantom{$0$}};
                    \node (E) at (1,1) {\textcolor{red}{s}};
                    \draw (A) -- (B) -- (C) -- (D) -- (A) -- (C);
                    \draw (B) -- (D);
             \end{tikzpicture}, \begin{tikzpicture}[scale=0.6, transform shape]
                 \node[shape=circle,draw=black] (A) at (0,0) {$2$};
                    \node[shape=circle,draw=black] (B) at (1,0) {$1$};
                    \node[shape=circle,draw=black] (C) at (0,1) {$2$};
                    \node[shape=circle,draw=red] (D) at (1,1) {\phantom{$0$}};
                    \node (E) at (1,1) {\textcolor{red}{s}};
                    \draw (A) -- (B) -- (C) -- (D) -- (A) -- (C);
                    \draw (B) -- (D);
             \end{tikzpicture}
    \\ \hline
    added
    & \begin{tabular}{c|c|c}
          1 & 2 & 3 
          \end{tabular}
    & \begin{tabular}{c|c|c}
          1 & 2 & 3 
          \end{tabular}
    & \begin{tabular}{c|c|c}
          1 & 2 & 3 
          \end{tabular}\\ \hline
    collected 
    & \begin{tabular}{c|c|c}
          2 & 2 & 4 
          \end{tabular}
    & \begin{tabular}{c|c|c}
          0 & 0 & 3 
          \end{tabular}
    & \begin{tabular}{c|c|c}
          3 & 4 & 4 
          \end{tabular}
\end{tabular}
    \caption{The number of collected particles in the sink for all recurrent configurations of $K_4$ depending on the number of added particles in the bottom left vertex.}
    \label{tab:k4-transPob}
\end{table}
Since all recurrent configurations in Table \ref{tab:k4-transPob} are equally likely, the transition matrix of $(X_i)_{i\in \N_0}$ is given by
\begin{align*}
    P = \begin{pmatrix}
        1 & 0 & 0 & 0 & 0 \\
        1/2 & 3/16 & 2/16 & 3/16 & 0 \\
        3/16 & 3/16 & 1/4 & 3/16 & 3/16 \\
        0 & 3/16 & 2/16 & 3/16 & 1/2 \\
        0 & 0 & 0 & 0 & 1
    \end{pmatrix}.
\end{align*}
The proof of the main result is based on investigating the absorbing states of the chain $(X_i)_{i\in\N}$.
\begin{proof}[Proof of Theorem \ref{thm:vicsek-explosion}]
We sample $\eta$ from to the $\mathsf{IVL}$ measure $\mu$, we add one particle at $o$ and consider the corresponding Markov chain $(X_i)_{i\in \N_0}$ starting with $X_0=1$. The starting constraint of $X_0=1$ means the addition of a single particle at $o=(0,0)$ to $\eta$. Recall that $X_i$ gives the number of particles stopped at the sink $(i,i)$ 
during the stabilization of $\eta+\delta_o$ on $\cup_{j=1}^iK^j$. Since during the stabilization of $\eta+\delta_o$ the sandpile on subgraphs rooted at 1-offset diagonal vertices stays invariant (thus stable), and on $\Vic\setminus\cup_{j=1}^iK^j $ the sandpile $\eta+\delta_o$ is already stable, this implies that the event of stabilization of $\eta+\delta_o$ is equivalent to $(X_i)_{i\in\N_0}$ ever reaching the state $0$. For $x\in\{0,1,2,3,4\}$, denote by $\tau_x$ the first time the Markov chain $(X_i)_{i\in \N}$ reaches the state $x$, that is $\tau_x$ is given by
$$\tau_x = \inf\{i\geq 0 | X_i = x \},$$ 
and let $x_k = \Prob(\tau_0 < \tau_4 \ | \ X_0 = k)$ for $k=0,\ldots,4$. Clearly it holds that $x_0=1$ and $x_4=0$. Denote by $p_{kj}=\Prob(X_{i+1}=j|X_i=k)$ the one-step transition probabilities of the chain $(X_i)_{i\in\N}$, for $j,k\in \{0,1,2,3,4\}$.
Factorizing with respect to the first step  yields 
\begin{align*}
 x_k = \sum_j p_{kj} \Prob(\tau_0 < \tau_4 \ | \ X_1 = j) = \sum_j p_{kj} x_j, \quad \text{for } k=1,2,3,
\end{align*}
and this equation can be written in matrix form as 
\begin{align*}
 (P-I)x = 0,
\end{align*}
where the vector $x$ is given by $x=(x_0,\ldots,x_4)^T$. The unique non-zero solution of the previous equation is given by $x=(1,\ 3/4,\ 1/2, \ 1/4,\ 0)^T$. Using that
 \begin{align*}
\mu(\eta + \delta_0 \text{ stabilizes}) = \Prob(\tau_0 < \tau_4 \ | \ X_0=1) = \frac{3}{4},
\end{align*}
yields the claim.
\end{proof}
Finally, taking the $k$-th powers of the transition matrix $P$ yields the $k$-step transition probabilities of $(X_i)_{i\in\N_0}$: for $k\in \N$ 
\begin{align*}
    \begin{pmatrix}
        \Prob(X_k=0 | X_0 = 1) \\
        \Prob(X_k=1 | X_0 = 1) \\
        \Prob(X_k=2 | X_0 = 1) \\
        \Prob(X_k=3 | X_0 = 1) \\
        \Prob(X_k=4 | X_0 = 1)
    \end{pmatrix}&=            
     \begin{pmatrix}
         \frac{3}{4} & -\frac{13 + 3 \sqrt{13}}{52} & - \frac{13 - 3 \sqrt{13}}{52} \\
         0 & \frac{13+\sqrt{13}}{52} & \frac{13-\sqrt{13}}{52} \\
         0 & \frac{4\sqrt{13}}{52} & -\frac{4\sqrt{13}}{52} \\
         0 & \frac{13+\sqrt{13}}{52} & -\frac{13-\sqrt{13}}{52} \\
         \frac{1}{4} & -\frac{13+3\sqrt{13}}{52} & \frac{13-3\sqrt{13}}{52}
     \end{pmatrix}\begin{pmatrix}
         1 \\
         \Big(\frac{5 + \sqrt{13}}{16}\Big)^k \\
         \Big(\frac{5 - \sqrt{13}}{16}\Big)^k
     \end{pmatrix},\\
    \begin{pmatrix}
        \Prob(X_k=0 | X_0 = 2) \\
        \Prob(X_k=1 | X_0 = 2) \\
        \Prob(X_k=2 | X_0 = 2) \\
        \Prob(X_k=3 | X_0 = 2) \\
        \Prob(X_k=4 | X_0 = 2)
    \end{pmatrix}&=            
     \begin{pmatrix}
         \frac{1}{2} & -\frac{13 + 5 \sqrt{13}}{52} & - \frac{13 - 5 \sqrt{13}}{52} \\
         0 & \frac{6\sqrt{13}}{52} & -\frac{6\sqrt{13}}{52} \\
         0 & 2\frac{13-\sqrt{13}}{52} & 2\frac{13+\sqrt{13}}{52}  \\
         0 & \frac{6\sqrt{13}}{52} & -\frac{6\sqrt{13}}{52} \\
         \frac{1}{2} & -\frac{13 + 5 \sqrt{13}}{52} & - \frac{13 - 5 \sqrt{13}}{52}
     \end{pmatrix}\begin{pmatrix}
         1 \\
         \Big(\frac{5 + \sqrt{13}}{16}\Big)^k \\
         \Big(\frac{5 - \sqrt{13}}{16}\Big)^k
     \end{pmatrix}.
\end{align*}
By the symmetry of $P$, we have $\Prob(X_k = i | X_0=3) = \Prob(X_k = 4-i | X_0=1)$, and using that $0$ and $4$ are absorbing states, we actually have obtained above all $k$-step transition probabilities.

The next result captures the exact distribution of the diameter of the avalanche  $\mathsf{T}(\eta+\delta_o)$ as a function of the transition probabilities of $(X_i)_{i\in \N_0}$. In order to simplify the notation,  for any $n\in \N_0$ with ternary representation $(a_i)_{i\in \N_0}$ we denote by 
$$\kappa_n=\sum_{i=0}^\infty \min(1,a_i)\cdot 3^i.$$

\begin{prop}
Let $\Vic$ be the infinite Vicsek graph and $\eta$ a stable configuration on $\Vic$ sampled from the $\mathsf{IVL}$  measure $\mu$. Then the diameter $\mathsf{T}(\eta+\delta_o)$ of the avalanche has the following distribution: for any $n\in \N_0$ we have
\begin{itemize}
\item If $n=0$, then
$$\mu(\mathsf{diam}(\mathsf{T}(\eta+\delta_o)) = 0)=\frac{137}{256}.$$
\item If there exists $i\in \N_0$ with $a_i=2$, then
$$\mu(\mathsf{diam}(\mathsf{T}(\eta+\delta_o)) = n)=0.$$
\item For all the other cases of $n\in\N$:
\begin{align*}
\mu(\mathsf{diam}(\mathsf{T}(\eta+\delta_0)) = n)&=
\Prob( X_{\kappa_{n-1}+1}\in \{2,3\}, X_{n+1} \in \{1,0\}, X_{n+2} = 0)\\
& + \Prob( X_{\kappa_{n-1}+1}=1 , X_{\kappa_{n-1}+2} \in \{1,2,3\}, X_{n+1} \in \{1,0\}, X_{n+2} = 0).
\end{align*}
\end{itemize}
\end{prop}
\begin{proof}
If $n=0$, then $\{\mathsf{diam}(\mathsf{T}(\eta+\delta_o)) = 0\}$ represents the event that none of the three
vertices $(1,0),(0,1),(1,1)$ topples when adding a particle at $o=(0,0)$ to $\eta$, that is $\eta+\delta_o$ is either stable, or it is unstable and stabilizes after one toppling; so  $(0,0)$ topples by sending one chip to the three neighbours which still remain stable, since the diameter of the toppled sites stays zero. In terms of the Markov chain $(X_i)_{i\in\N}$, $\{\mathsf{diam}(\mathsf{T}(\eta+\delta_o)) = 0\}$ corresponds to the event that given $\{X_0=1\}$ we have either $\{X_1 = 0\}$ or $\{X_1 = 1, X_2 = 0\}$ and neither (1,0) nor (0,1) topple which has probability $p_{10}+\frac{3}{8}p_{11}p_{10}=\frac{1}{2}+\frac{3}{8}\cdot\frac{1}{2}\cdot\frac{3}{16}=\frac{137}{256}$.

For the following let $N = {\lceil \log_3(n) \rceil }$. The ternary representation of $n$ determines the position of the vertex $(n,n)$ in the finite Vicsek graph $\Vic_N$: if $a_{N-1}=1$, the vertex is located in the center copy of $\Vic_{N-1}$ in $\Vic_N$, if $a_{N-1}=2$, then the vertex is in the top right copy of $\Vic_{N-1}$ in $\Vic_N$.
Now let us consider the second case of the proposition, i.e.~for the ternary representation $\sum_{i=0}^\infty a_i\cdot 3^i$ of $n$, there exists some $i\in\N$ such that $a_i=2$, and let us fix the largest such $i\in\N$. If a vertex $x$ that is not on the diagonal and also not in the corner copies of $K_4$ topples, then since all of its descendants must topple, we would obtain a diameter that is larger then the distance of $x$ to the origin. Thus it suffices to only consider vertices on the diagonal. The vertices at distance $n$ of the origin on the diagonal are then given by $(n,n),(n-1,n)$ and $(n,n-1)$. Notice that since $(n-1,n)$ and $(n,n-1)$ are descendants of $(n,n)$, every time $(n,n)$ topples so must $(n-1,n)$ and $(n,n-1)$. However, if $(n-1,n)$ topples then also all its descendants must topple. For an illustration of this argument, assume that the vertex $x$ in Figure \ref{fig:vicsekLev2} topples, then also all of the vertices coloured in red must topple. However, if all those vertices topple, then the diameter of the set of toppled sites must be at least $m+3^{i+1}$, where $m$ has ternary representation
$$b_j=\begin{cases}a_j,&j>i\\0,&j\leq i\end{cases}.$$    
By symmetry the same holds true when $(n,n-1)$ topples. Thus, since $m+3^{i+1}>n$, the diameter cannot be equal to $n$, showing the second assertion. 
    
For the last case, let $m=\min\{ i \in \N_0 : a_i = 1\}$, then the diagonal vertex at distance $n$ is the corner point of a copy of $\Vic_m$. Let us first consider the case $m\geq 1$. Then we can already reach a diameter of $n$ if the center vertex of the corresponding copy of $\Vic_m$ topples, as argued in the previous case. By the center or midpoint of $\Vic_m$ we mean the vertex with coordinates $(\frac{3^m+1}{2},\frac{3^m+1}{2})$. This center vertex of the corresponding copy of $\Vic_m$, of which the diagonal vertex at distance $n$ is a corner of, is given by $(\kappa_{n-1},\kappa_{n-1})$, and this vertex topples if either $X_{\kappa_{n-1}+1} \in \{ 2,3\}$ or $X_{\kappa_{n-1}+1}=1$ and $X_{\kappa_{n-1}+2} \in \{1,2,3\}$. If furthermore the vertex at distance $n+1$ does not topple, that is if $X_{n+1}=0$ or $X_{n+1}=1$ and $X_{n+2}=0$, this is equivalent to the event that vertices at distance $n$ are the furthest that toppled during stabilization. In the case $m=0$, the statement is true by noticing that $\kappa_{n-1}=n-1$. This can be seen from the fact that the ternary representation of $n-1$ also only has $a_i=0$ or $a_i=1$ for all $i\in\N_{0}$ if $m=0$, thus $\kappa_{n-1}=n-1$ from the definition of $\kappa_{n-1}$.
\end{proof}

With this result and the $k$-step transition probabilities of the Markov chain 
$(X_n)_{n\in\N}$, we can graphically represent the distribution of the avalanche radii as in Figure \ref{fig:dist_av_rad}. The logarithmic scale on the $y$-axis shows the exponential decay in probability and the gaps capture the jumps of avalanche sizes, once the midpoint $(\frac{3^n+1}{2},\frac{3^n+1}{2})$ of $\Vic_n$ topples. Also notice the self similar structure of the gaps originating from the  self similarity of the underlying state space.
\begin{figure}[h]
    \centering
    \begin{tikzpicture}
    \pgfplotsset{
    width=.8\textwidth}

    \begin{axis}[ybar, ymode = log, log origin=infty, bar width=0.5pt,
    xtick={1, 10, 27, 40, 90,120,245},
    xtick pos=left,
    ytick pos=left]
    \addplot coordinates{
    (1,9/128)(3,513/8192)(4,189/16384)(9,113759235/8589934592)(10,1199637/4294967296)(12,1042963857/4398046511104)(13,382257441/8796093022208)(27,2021636242128092788737/38685626227668133590597632)(28,306567753658951869/77371252455336267181195264)(30,266529863320182607149/79228162514264337593543950336)(31,97686053815163514237/158456325028528675187087900672)(36,459260096086157873092299/649037107316853453566312041152512)(37,619904475572423732084241/41538374868278621028243970633760768)(39,33684041531993129487509697/2658455991569831745807614120560689152)(40,12345562530280703719909341/5316911983139663491615228241121378304)(81,2526284577241797296683285833067975052082012969887917611519833811/904625697166532776746648320380374280103671755200316906558262375061821325312)(82,5116303807232202706943481457986859281253291105014357/452312848583266388373324160190187140051835877600158453279131187530910662656)(84,4448112164996809250154643186115189278752427690840041057/463168356949264781694283940034751631413079938662562256157830336031652518559744)(85,1630280820741555079742049364725914709491612782904691441/926336713898529563388567880069503262826159877325124512315660672063305037119488)(90,15329167207386258089885162428123750897269445650887944300449/7588550360256754183279148073529370729071901715047420004889892225542594864082845696)(91,10345574805691444291808088921601247979400019830304274509013/242833611528216133864932738352939863330300854881517440156476551217363035650651062272)(93,562152372114208404860735912888052955138805190212421300162441/15541351137805832567355695254588151253139254712417116170014499277911234281641667985408)(94,206034874256103918368874433573055838065290959352532512268373/31082702275611665134711390509176302506278509424834232340028998555822468563283335970816)(108,17434431323234278775895352847630055456672652253572363862170575132554188161/2187250724783011924372502227117621365353169430893212436425770606409952999199375923223513177023053824)(109,2643816100888923701041810308197152837905435977833317249689241605423841/4374501449566023848745004454235242730706338861786424872851541212819905998398751846447026354046107648)(111,143658281000570667501593686175924651083545352553454933128286678436140417/279968092772225526319680285071055534765205687154331191862498637620473983897520118172609686658950889472)(112,52652300924183034090926738127531184875728619237703897386288850610078301/559936185544451052639360570142111069530411374308662383724997275240947967795040236345219373317901778944)(117,31684982442847302504998196255279873442161735967644771156374012731485541392075/293567822846729153486185074598667128421960318613539983838411371441526128139326055432962374798096087878991872)(118,334125453095338457462913524520329778680549554973856946727986654471100138133/146783911423364576743092537299333564210980159306769991919205685720763064069663027716481187399048043939495936)(120,290488514471633597285131459261551967734979477104509710599074913926796528861993/150306725297525326584926758194517569752043683130132471725266622178061377607334940381676735896625196994043838464)(121,106467156452460912998964780308798335464868395384248301767813217218611321662009/300613450595050653169853516389035139504087366260264943450533244356122755214669880763353471793250393988087676928)(243,19246687987621817376495094066653394229420347715595538710990554535506334999143989532213013866670188736419591065349679489316567022920605428766055005806792938710827964493548598374988236130713/45184223339331479951185741475274045813621662589625240394934430893803101285779175998493982735923679951534365847972543945249972749854054255162024849415791137702892737463723457929162113159883256443580117751661352485851758592)(244,23781863680670203206624068234843090522332466703865185048732069956650401211907207398037954206321100320616884999109930722067460318584622799677600424000300989/90368446678662959902371482950548091627243325179250480789868861787606202571558351996987965471847359903068731695945087890499945499708108510324049698831582275405785474927446915858324226319766512887160235503322704971703517184)(246,20675941290810824346775222223611250033640232318212393319106741946478478686088614107958946172259596630258450934047697973028201086491783728607415641052876528349/92537289398950870940028398541361245826297164983552492328825714470508751433275752444915676643171696540742381256647769999871944191701103114571826891603540250015524326325705641838924007751440909196452081155402449891024401596416)(247,7577954261684284175413642183491643832170495832023175883829310157514032606797919855015998957658789134261138913902931267041683335021051226925595394242272328237/185074578797901741880056797082722491652594329967104984657651428941017502866551504889831353286343393081484762513295539999743888383402206229143653783207080500031048652651411283677848015502881818392904162310804899782048803192832)(252,35626907490221544548879447287368285943599633855587648748666213500458527760469273657684966292003436014266497457006327861246542793348956128202401451790168397734889/758065474756205534740712640850831325809026375545262017157740252942407691741394964028749223060862538061761587254458531838950966818415436714572405896016201728127175281260180617944465471499803928137335448825056869507271897877839872)(253,48088827207512952494205909685843201840309617425162304708823148549411470336652264964097913432816848845876108411911071163417985036936588328207604979842638947309441/48516190384397154223405609014453204851777688034896769098095376188314092271449277697839950275895202435952741584285346037692861876378587949732633977345036910600139218000651559548445790175987451400789468724803639648465401464181751808)(255,31693507737097536245807288193350834051390094559052062019546692906548171059573301934472134772105479196225979650914557904834707021418050334975386963336417339815728735747124792473788636984695572522586337123322631593837550417410886353557315466002977073276497936850839179556420361727119259039273752110205802530830646820264945498189946501932747340868166202637202178960817698806485908769573265172703/37661131670641134196564321969192408014251456999151389864994922485394646743621925010454453501677282025991386973138875074810350265293671223827650202935793353928987708427837894137043469967463773432450711611521684688399958529029759361085334173782804811953090869395973448087960521176275272318934283811996654224537015142331587374110222375180668100782509056060170328026795191929322603612211172871951146487222820410109263092167280797070716301943769265419640699380301824)
    };
    \end{axis}
    \end{tikzpicture}
    \caption{Distribution of the avalanche radius.}
    \label{fig:dist_av_rad}
\end{figure}
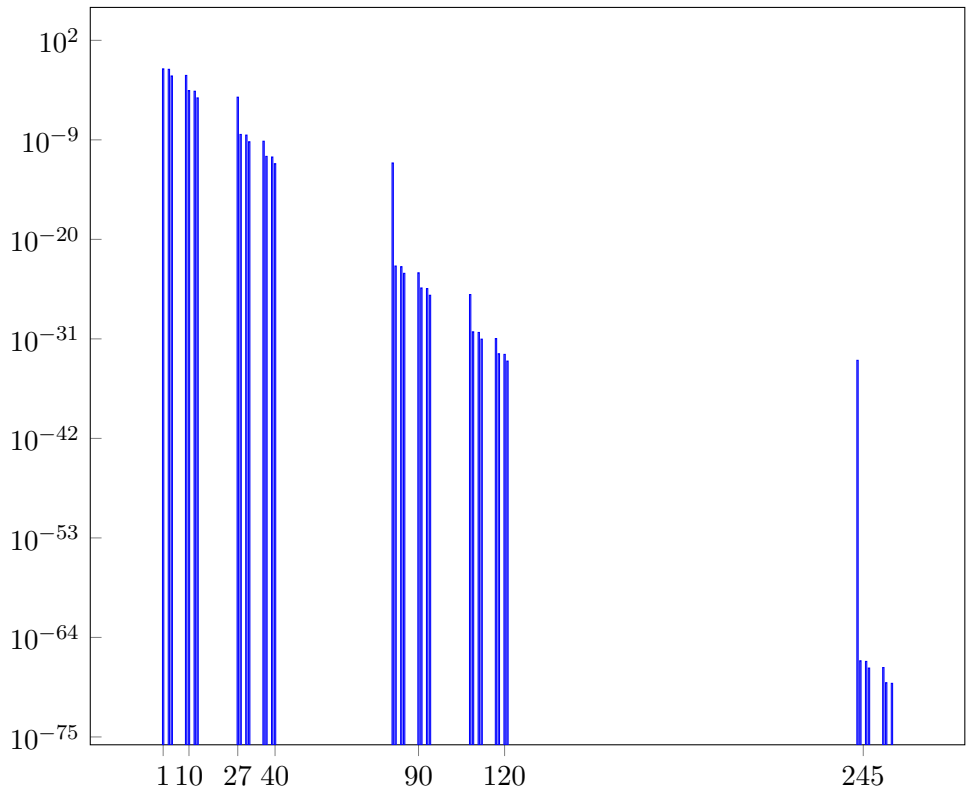

\section{Order of recurrent sandpiles and stabilization}\label{sec:relation-order-stabilization}

So far we have shown that the infinite Vicsek graph $\Vic$ is an example of a recurrent graph on which the infinite volume limit plus one particle stabilizes with probability  $\frac{3}{4}$.  However, other well-known examples of recurrent graphs, such as $\Z^2$ and the Sierpi\'nski gasket graphs, do not seem to share this behaviour, as one expects from simulations  in \cite{jarai_toppling_2019,daerden_sandpiles_1998}. Thus, it seems natural to ask which qualitative property of the Vicsek graph is responsible for the non-stabilization behaviour in infinite volume limit.
A closer look at the structure of the sandpile group of the finite graphs $\Vic_n$, for $n\in\N$, reveals that we can find an uniform upper bound on the order of all elements of the sandpile group on $\Vic_n$; see Section \ref{sec:sandpile-group} for details. Below we prove that such an uniform upper bound on the order of recurrent sandpiles is sufficient to prove non-stabilization in infinite volume.

In the next result, $H$ is a finite graph with sink $s$, and we denote by $\mathcal{R}_{H}$ its sandpile group, i.e.~the set of recurrent states of the sandpile Markov chain over $H$.

\begin{lem}\label{lem:send-to-sink-upper-bound}
Let $H$ be a finite connected graph with designated sink vertex $s$, and let $x$ be any other vertex. If there exists $\varepsilon>0$ such that
\begin{align*}
\mathbb{P}(\eta+\delta_x\text{ sends at least one particle to $s$ during stabilization})<\varepsilon,
\end{align*}
where $\eta$ is sampled from the uniform distribution over $\mathcal{R}_H$, then it holds
\begin{align*}
\mathbb{P}(\eta+2\delta_x\text{ sends at least one particle to $s$ during stabilization})<2\varepsilon.
\end{align*}
\end{lem}
\begin{proof}
Let
\begin{align*}
 A=\{\eta\in\mathcal{R}_H:\ \eta+\delta_x\text{ sends at least one particle to $s$ during stabilization}\},
\end{align*}
and 
\begin{align*}
B=\{\eta\in\mathcal{R}_H:\ \eta+2\delta_x\text{ sends at least one particle to $s$ during stabilization}\},
\end{align*}
and pick a sandpile $\eta\in B$. Then either $\eta\in A$ or $\eta\notin A$, but $(\eta+\delta_x)^\circ\in A$. Either way, we have
\begin{align*}
\mathbb{P}(\eta\in B)=\mathbb{P}(\eta\in A)+\mathbb{P}(\eta\notin A,\eta+\delta_x\in A)
\leq \mathbb{P}(\eta\in A)+\mathbb{P}(\eta+\delta_x\in A)<2\varepsilon,
\end{align*}
and this completes the proof.
\end{proof}
\begin{lem}\label{lem:almost-surely-to-sink}
Let $H$ be a finite connected graph with sink vertex $s$, and let $x$ be any other vertex. If there exists $M\in\N$ such that $\order([\eta])<M$ for all $\eta\in\mathcal{R}_H$, then we have
\begin{align*}
\mathbb{P}(\eta+M\delta_x\text{ sends at least one particle to $s$ during stabilization})=1,
\end{align*}
where $\eta$ is sampled from the uniform distribution over $\mathcal{R}_H$.
\end{lem}
\begin{proof}
 We have $(\eta+\order([\delta_x])\delta_x)^\circ=\eta$ and thus
\begin{align*}
 \mathbb{P}(\eta+\order([\delta_x])\delta_x\text{ sends at least one particle to $s$ during stabilization})=1.
\end{align*}
 The claim now follows from $\order([\delta_x])<M$ together with monotonicity.
\end{proof}
We can now prove Theorem \ref{thm:relation-order-stabilziation}.
\begin{proof}[Proof of Theorem \ref{thm:relation-order-stabilziation}]
Assume the statement is false, thus for every $\varepsilon>0$ we can find $R>0$ such that
 \begin{align*}
  \mu(\mathsf{diam}(\mathsf{T}(\eta+\delta_x)>R)<\varepsilon.
 \end{align*}
If $\mu_n$ is the uniform distribution over $\mathcal{R}_n$, and we let $\varepsilon = 2^{-M-1}$ and choose $n$ large enough, we get
\begin{align*}
\mu_n(\mathsf{diam}(\mathsf{T}(\eta+\delta_x)>R)<2^{-M},
\end{align*}
and by inductively applying Lemma \ref{lem:send-to-sink-upper-bound} we obtain
\begin{align*}
 \mu_n(\mathsf{diam}(\mathsf{T}(\eta+M\delta_x)>R)<2^M\cdot 2^{-M}=1,
 \end{align*}
 which contradicts the statement of Lemma \ref{lem:almost-surely-to-sink} applied to $G_n$. Thus the assumption is false and the probabilities $\mu(\mathsf{diam}(\mathsf{T}(\eta+\delta_x))>R)$ are uniformly bounded away from $0$ by some constant $c>0$.
\end{proof}

\section{Sandpile group of the Vicsek graphs}\label{sec:vicsek-sandpile-group}

Lemma \ref{lem:4-particles-to-origin} and Corollary \ref{cor:4particlesDiag} show that for any recurrent sandpile $\eta\in\mathcal{R}_n$, if one adds $4$ particles to any vertex $x$ on the diagonal set $D_n^0$  of $\Vic_n$ and then stabilizes $\eta+4\delta_x$, the stabilized configuration is again $\eta$, so $\eta$ stays invariant under the stabilisation operation on $\Vic_n$ with sink $s_n=(3^n,3^n)$. In terms of the sandpile group $\mathcal{R}_n$ of $\Vic_n$, for any $n\in\N$ this means that
\begin{align*}
    \forall x\in \diag_n^0 \text{ it holds}\quad  4\cdot[\delta_x]=[0],
\end{align*}
that is, the recurrent sandpile in the equivalence class of $\delta_x$ has order $4$ in the sandpile group of $\Vic_n$, thus we suspect that $4$ might be related with the order of the elements in the sandpile group $\mathcal{R}_n$. This is partially the case as the next result shows. 
\begin{thm}\label{thm:sandpile-group}
For all $n\in \N$ and all recurrent sandpiles $\eta\in\mathcal{R}_{n}$ we have
 \begin{align*}
  \order_{\mathcal{R}_{n}}([\eta])\in\{2,4\} \quad \text{and} \quad \mathcal{R}_{n}\cong \Z_4^{2\cdot 5^n}.
    \end{align*}
\end{thm}
The proof of Theorem \ref{thm:sandpile-group} will be completed in the following three steps:
\begin{enumerate}
\item We remark that any recurrent sandpile $\eta$ on $\Vic_n$ must have an order that divides the number of recurrent sandpiles. Since by the matrix-tree theorem, $|\mathcal{R}_{n}|=16^{5^n}$, for any $\eta\in\mathcal{R}_n$, there must exist a $k\in\N$ such that
    \begin{align*}
        \order_{\mathcal{R}_{n}}([\eta])=2^k.
    \end{align*}

    \item The second step is to prove that for all $x\in V_n$ we have
    \begin{align*}
        \order_{\mathcal{R}_{n}}([\delta_x])\leq4,
    \end{align*}
    from which follows that all recurrent sandpiles have order at most $4$.
\item Finally, we show the second assertion of Theorem \ref{thm:sandpile-group} by counting the sandpiles of order $2$.
\end{enumerate}

\subsection{Upper bounds on the order of group elements}

We show here that any recurrent sandpile on $\Vic_n$ has order at most $4$, by first considering the equivalence classes of $\delta_x$ for all $x\in V_n$.
For $x\in V_n$, denote by $\gamma_x^{(n)}$  the geodesic from $x$ to the sink $s_n=(3^n,3^n)$, i.e.~the shortest path in $\Vic_n$ connecting $x$ to $s_n$. We define the set of descendants of $x$ in $\Vic_n$ by
\begin{align*}
    \descset_n(x)=\{v\in V_n:\ \gamma_v^{(n)}\text{ goes through }x\} \backslash \{x\},
\end{align*}
and the \textit{geodesic subgraph} $\geodsubgraph_n(x)$  of $x$ as the subgraph of $\Vic_n$ with vertex set
\begin{align*}
    V(\geodsubgraph_n(x))=\{v \in V_n: d(v,\gamma_x^{(n)}) \leq 1 \} \backslash\descset_n(x).
\end{align*}
Visually speaking, $\geodsubgraph_n(x)$ is the path of successive copies of complete graphs $K_4$ that connects $x$ with the sink vertex.
For vertices $x\in\diag_n^1$ we have $\descset_n(x) = T^x_n$ and $\Gamma_n(x)$ is the subgraph of $\Vic_n$ with vertex set $V(\geodsubgraph_n(x))= \{y\in \diag_n : \|y\|_\infty \geq \|x\|_\infty -1\}$. It follows immediately from Lemma \ref{lem:diagonalization} by exploiting the cut point structure of the Vicsek graphs and invoking Dhar's burning algorithm, that by adding a particle to the vertex $x$, we can change a recurrent sandpile only along the geodesic subgraph of $x$.
\begin{figure}[t]
    \centering
    \begin{tikzpicture}[scale=0.75]
        \node[shape=circle,draw=black] (A) at (0,0) {};
        \node[shape=circle,draw=black] (B) at (1,0) {};
        \node[shape=circle,draw=black] (C) at (0,1) {};
        \node[shape=circle,draw=black] (D) at (1,1) {};
        \draw[color=black] (A) -- (B) -- (C) -- (D) -- (A) -- (C);
        \draw[color=black] (B) -- (D);
         \node (A) at (0,0) {\small\color{blue} $o$};
        \begin{scope}[shift={(1,1)}]
            \node[shape=circle,draw=black] (A) at (0,0) {};
            \node[shape=circle,draw=black] (B) at (1,0) {};
            \node[shape=circle,draw=black] (C) at (0,1) {};
            \node[shape=circle,draw=black] (D) at (1,1) {};
            \draw[color=black] (A) -- (B) -- (C) -- (D) -- (A) -- (C);
            \draw[color=black] (B) -- (D);
        \end{scope}
        \begin{scope}[shift={(2,2)}]
            \node[shape=circle,draw=black] (A) at (0,0) {};
            \node[shape=circle,draw=black] (B) at (1,0) {};
            \node[shape=circle,draw=black] (C) at (0,1) {};
            \node[shape=circle,draw=orange] (D) at (1,1) {};
            \draw[color=black] (A) -- (B) -- (C) -- (D) -- (A) -- (C);
            \draw[color=black] (B) -- (D);
        \end{scope}
        \begin{scope}[shift={(2,0)}]
            \node[shape=circle,draw=black] (A) at (0,0) {};
            \node[shape=circle,draw=black] (B) at (1,0) {};
            \node[shape=circle,draw=black] (C) at (0,1) {};
            \node[shape=circle,draw=black] (D) at (1,1) {};
            \draw (A) -- (B) -- (C) -- (D) -- (A) -- (C);
            \draw (B) -- (D);
        \end{scope}
        \begin{scope}[shift={(0,2)}]
            \node[shape=circle,draw=black] (A) at (0,0) {};
            \node[shape=circle,draw=black] (B) at (1,0) {};
            \node[shape=circle,draw=black] (C) at (0,1) {};
            \node[shape=circle,draw=black] (D) at (1,1) {};
            \draw (A) -- (B) -- (C) -- (D) -- (A) -- (C);
            \draw (B) -- (D);
        \end{scope}
        \begin{scope}[shift={(3,3)}]
            \node[shape=circle,draw=black] (A) at (0,0) {};
            \node[shape=circle,draw=black] (B) at (1,0) {};
            \node[shape=circle,draw=black] (C) at (0,1) {};
            \node[shape=circle,draw=orange] (D) at (1,1) {};
            \draw (A) -- (B) -- (C) -- (D) -- (A) -- (C);
            \draw (B) -- (D);
            \begin{scope}[shift={(1,1)}]
                \node[shape=circle,draw=orange] (A) at (0,0) {};
                \node[shape=circle,draw=orange] (B) at (1,0) {};
                \node[shape=circle,draw=orange] (C) at (0,1) {};
       
                \node[shape=circle,draw=orange] (D) at (1,1) {};
                \draw[color=orange] (A) -- (B) -- (C) -- (D) -- (A) -- (C);
                \draw[color=orange] (B) -- (D);
            \end{scope}
            \begin{scope}[shift={(2,2)}]
                \node[shape=circle,draw=orange] (A) at (0,0) {};
                \node[shape=circle,draw=orange] (B) at (1,0) {};
                \node[shape=circle,draw=orange] (C) at (0,1) {};
                \node[shape=circle,draw=orange] (D) at (1,1) {};
                \draw[color=orange] (A) -- (B) -- (C) -- (D) -- (A) -- (C);
                \draw[color=orange] (B) -- (D);
            \end{scope}
            \begin{scope}[shift={(2,0)}]
                \node[shape=circle,draw=orange] (A) at (0,0) {};
                \node[shape=circle,draw=orange] (B) at (1,0) {};
                \node[shape=circle,draw=orange] (C) at (0,1) {};
                \node[shape=circle,draw=orange] (D) at (1,1) {};
                \draw[draw=orange] (A) -- (B) -- (C) -- (D) -- (A) -- (C);
                \draw[draw=orange] (B) -- (D);
            \end{scope}
            \begin{scope}[shift={(0,2)}]
                \node[shape=circle,draw=black] (A) at (0,0) {};
                \node[shape=circle,draw=orange] (B) at (1,0) {};
                \node[shape=circle,draw=black] (C) at (0,1) {};
                \node[shape=circle,draw=black] (D) at (1,1) {};
                \draw (A) -- (B) -- (C) -- (D) -- (A) -- (C);
                \draw (B) -- (D);
            \end{scope}
        \end{scope}
        \begin{scope}[shift={(6,6)}]
            \node[shape=circle,draw=orange] (A) at (0,0) {};
            \node[shape=circle,draw=orange] (B) at (1,0) {};
            \node[shape=circle,draw=orange] (C) at (0,1) {};
            \node[shape=circle,draw=orange] (D) at (1,1) {};
            \draw[color=orange] (A) -- (B) -- (C) -- (D) -- (A) -- (C);
            \draw[color=orange] (B) -- (D);
            \begin{scope}[shift={(1,1)}]
                \node[shape=circle,draw=orange] (A) at (0,0) {};
                \node[shape=circle,draw=orange] (B) at (1,0) {};
                \node[shape=circle,draw=orange] (C) at (0,1) {};
                \node[shape=circle,draw=orange] (D) at (1,1) {};
                \draw[color=orange] (A) -- (B) -- (C) -- (D) -- (A) -- (C);
                \draw[color=orange] (B) -- (D);
            \end{scope}
            \begin{scope}[shift={(2,2)}]
                \node[shape=circle,draw=orange] (A) at (0,0) {};
                \node[shape=circle,draw=orange] (B) at (1,0) {};
                \node[shape=circle,draw=orange] (C) at (0,1) {};
                \node[shape=circle,draw=orange] (D) at (1,1) {};
                \draw[color=orange] (A) -- (B) -- (C) -- (D) -- (A) -- (C);
                \draw[color=orange] (B) -- (D);
            \end{scope}
            \begin{scope}[shift={(2,0)}]
                \node[shape=circle,draw=black] (A) at (0,0) {};
                \node[shape=circle,draw=black] (B) at (1,0) {};
                \node[shape=circle,draw=orange] (C) at (0,1) {};
                \node[shape=circle,draw=black] (D) at (1,1) {};
                \draw (A) -- (B) -- (C) -- (D) -- (A) -- (C);
                \draw (B) -- (D);
            \end{scope}
            \begin{scope}[shift={(0,2)}]
                \node[shape=circle,draw=black] (A) at (0,0) {};
                \node[shape=circle,draw=orange] (B) at (1,0) {};
                \node[shape=circle,draw=black] (C) at (0,1) {};
                \node[shape=circle,draw=black] (D) at (1,1) {};
                \draw (A) -- (B) -- (C) -- (D) -- (A) -- (C);
                \draw (B) -- (D);
            \end{scope}
        \end{scope}
        \begin{scope}[shift={(0,6)}]
            \node[shape=circle,draw=black] (A) at (0,0) {};
            \node[shape=circle,draw=black] (B) at (1,0) {};
            \node[shape=circle,draw=black] (C) at (0,1) {};
            \node[shape=circle,draw=black] (D) at (1,1) {};
            \draw[color=black] (A) -- (B) -- (C) -- (D) -- (A) -- (C);
            \draw[color=black] (B) -- (D);
            \begin{scope}[shift={(1,1)}]
                \node[shape=circle,draw=black] (A) at (0,0) {};
                \node[shape=circle,draw=black] (B) at (1,0) {};
                \node[shape=circle,draw=black] (C) at (0,1) {};
                \node[shape=circle,draw=black] (D) at (1,1) {};
                \draw[color=black] (A) -- (B) -- (C) -- (D) -- (A) -- (C);
                \draw[color=black] (B) -- (D);
            \end{scope}
            \begin{scope}[shift={(2,2)}]
                \node[shape=circle,draw=black] (A) at (0,0) {};
                \node[shape=circle,draw=black] (B) at (1,0) {};
                \node[shape=circle,draw=black] (C) at (0,1) {};
                \node[shape=circle,draw=black] (D) at (1,1) {};
                \draw[color=black] (A) -- (B) -- (C) -- (D) -- (A) -- (C);
                \draw[color=black] (B) -- (D);
            \end{scope}
            \begin{scope}[shift={(2,0)}]
                \node[shape=circle,draw=black] (A) at (0,0) {};
                \node[shape=circle,draw=black] (B) at (1,0) {};
                \node[shape=circle,draw=black] (C) at (0,1) {};
                \node[shape=circle,draw=black] (D) at (1,1) {};
                \draw[color=black] (A) -- (B) -- (C) -- (D) -- (A) -- (C);
                \draw[color=black] (B) -- (D);
            \end{scope}
            \begin{scope}[shift={(0,2)}]
                \node[shape=circle,draw=black] (A) at (0,0) {};
                \node[shape=circle,draw=black] (B) at (1,0) {};
                \node[shape=circle,draw=black] (C) at (0,1) {};
                \node[shape=circle,draw=black] (D) at (1,1) {};
                \draw[color=black] (A) -- (B) -- (C) -- (D) -- (A) -- (C);
                \draw[color=black] (B) -- (D);
            \end{scope}
        \end{scope}
        \begin{scope}[shift={(6,0)}]
            \node[shape=circle,draw=black] (A) at (0,0) {};
            \node[shape=circle,draw=black] (B) at (1,0) {};
            \node[shape=circle,draw=black] (C) at (0,1) {};
            \node[shape=circle,draw=black] (D) at (1,1) {};
            \draw (A) -- (B) -- (C) -- (D) -- (A) -- (C);
            \draw (B) -- (D);
            \begin{scope}[shift={(1,1)}]
                \node[shape=circle,draw=black] (A) at (0,0) {};
                \node[shape=circle,draw=black] (B) at (1,0) {};
                \node[shape=circle,draw=orange] (C) at (0,1) {};
                \node[shape=circle,draw=black] (D) at (1,1) {};
                \node (E) at (0,1) {\small\color{orange} $x$};
                \draw (A) -- (B) -- (C) -- (D) -- (A) -- (C);
                \draw (B) -- (D);
            \end{scope}
            \begin{scope}[shift={(2,2)}]
                \node[shape=circle,draw=black] (A) at (0,0) {};
                \node[shape=circle,draw=black] (B) at (1,0) {};
                \node[shape=circle,draw=black] (C) at (0,1) {};
                \node[shape=circle,draw=black] (D) at (1,1) {};
                \draw (A) -- (B) -- (C) -- (D) -- (A) -- (C);
                \draw (B) -- (D);
            \end{scope}
            \begin{scope}[shift={(2,0)}]
                \node[shape=circle,draw=black] (A) at (0,0) {};
                \node[shape=circle,draw=black] (B) at (1,0) {};
                \node[shape=circle,draw=black] (C) at (0,1) {};
                \node[shape=circle,draw=black] (D) at (1,1) {};
                \draw (A) -- (B) -- (C) -- (D) -- (A) -- (C);
                \draw (B) -- (D);
            \end{scope}
            \begin{scope}[shift={(0,2)}]
                \node[shape=circle,draw=orange] (A) at (0,0) {};
                \node[shape=circle,draw=orange] (B) at (1,0) {};
                \node[shape=circle,draw=orange] (C) at (0,1) {};
                \node[shape=circle,draw=orange] (D) at (1,1) {};
                \draw[draw=orange] (A) -- (B) -- (C) -- (D) -- (A) -- (C);
                \draw[draw=orange] (B) -- (D);
            \end{scope}
        \end{scope}
    \end{tikzpicture}
    \caption{The point $x=(7,2)$ in $\Vic_2$ and its geodesic subgraph $\Gamma_2(x)$ coloured in orange.}
    \label{fig:geodsubgraph}
\end{figure}
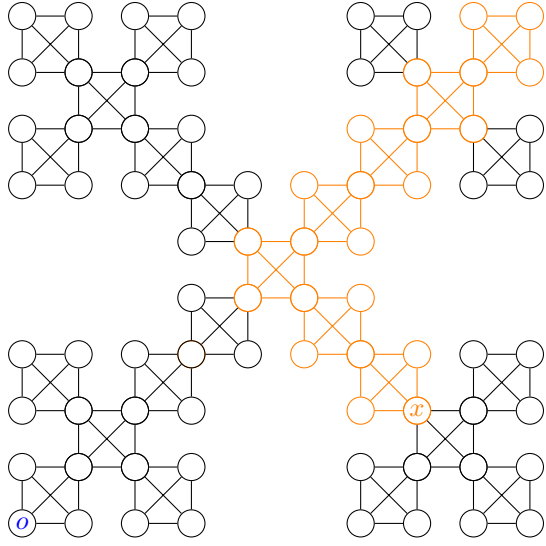
\begin{lem}\label{lem:stability-outside-geodsubgraph}
For any $n\in\N$, $x\in V_n$, and any recurrent sandpile $\eta$ on $\Vic_n$ we have
$$(\eta+\delta_x)^\circ|_{\Gamma_n(x)^{\complement}}=\eta|_{\Gamma_n(x)^{\complement}},$$
where $\Gamma_n(x)^{\complement}$ is the complement of $\Gamma_n(x)$.
\end{lem}
Next we upper bound the order of the equivalence classes $[\delta_x]$.
\begin{lem}\label{lem:order-upper-bound-delta}
For every $n\in\N$, $x\in V_n$ and all recurrent sandpiles $\eta$ on $\Vic_n$ we have
    \begin{align*}
        (\eta+4\cdot\delta_x)^\circ=\eta.
    \end{align*}
\end{lem}
\begin{proof}
 By Lemma \ref{lem:stability-outside-geodsubgraph}, it suffices to consider only recurrent sandpiles on the geodesic subgraph $\geodsubgraph_n(x)$. Then the proof works as in Lemma \ref{lem:4-particles-to-origin}.
\end{proof}
\begin{prop}\label{prop:order-upper-bound}
For every $n\in\N$ and any recurrent sandpile $\eta\in\mathcal{R}_n$ we have
    \begin{align*}
        \order_{\mathcal{R}_{n}}([\eta])\in \{2,4\}.
    \end{align*}
\end{prop}
\begin{proof}
In view of Lemma \ref{lem:order-upper-bound-delta}, for $x\in V_n$ it holds $4\cdot [\delta_x]=[0]$. Thus for $\eta\in\mathcal{R}_n$ we have
\begin{align*}
 4\cdot[\eta]&=\sum_{x\in V_n\backslash\{s_n\}}4\cdot\eta(x)\cdot[\delta_x]
 =\sum_{x\in V_n\backslash\{s_n\}}\eta(x)\big(4\cdot[\delta_x]\big)\\
  &=\sum_{x\in V_n\backslash\{s_n\}}\eta(x)\cdot[0]
  =\sum_{x\in V_n\backslash\{s\}}[0]=[0],
\end{align*}
and this finishes the proof.
\end{proof}

\subsection{The identity element of the sandpile group $\mathcal{R}_n$}

The identity element of the sandpile group has been investigated in a series of works \cite{dhar-asm-1995,caracciolo-identity-2008, borgne-identity-2002} and in \cite{Pegden_2013} the authors focus on the scaling limit of the single-source sandpile model  on $\Z^2$. The scaling limit of the identity element of the sandpile group on $\Z^2$ was characterized as a consequence of \cite{levine-pegden-smart-2016,levine-pegden-smart-2017} and \cite{pegden-smart-2020}. There is also an explicit formula for the identity element on the $F$-lattice given in \cite{paoletti2012} which was recently connected to the scaling limit in \cite{bou-rabee-2024}.
In \cite{sandpile-group-gasket, scaling_identity} the identity element  of the sandpile group and its scaling limit on Sierpi\'nski gasket graphs has been investigated. We analyse below the identities of the sandpile groups $\mathcal{R}_n$ over Vicsek graphs $\Vic_n$, $n\in \N$.

For every $n\in\N$, we denote the five copies of $\Vic_{n-1}$ that are used to construct $\Vic_n$ and the corresponding cutpoints  with their coordinates in $\mathbb{R}^2$ by
\begin{align*}
    \Vic_{n-1}^\mathsf{LB} &= \Vic_{n-1}, &&c_\mathsf{LB}=(3^{n-1},3^{n-1}), \\ 
    \Vic_{n-1}^\mathsf{RB} &= (2\cdot3^{n-1},0) + \Vic_{n-1}, &&c_\mathsf{RB}=(2\cdot 3^{n-1},3^{n-1}), \\
    \Vic_{n-1}^\mathsf{RT} &= (2\cdot3^{n-1},2\cdot3^{n-1}) + \Vic_{n-1}, &&c_\mathsf{RT}=(2\cdot 3^{n-1},2 \cdot 3^{n-1}), \\
    \Vic_{n-1}^\mathsf{LT} &= (0,2\cdot3^{n-1}) + \Vic_{n-1},&&c_\mathsf{LT}=(3^{n-1},2\cdot 3^{n-1}), \\
    \Vic_{n-1}^\mathsf{M} &= (3^{n-1},3^{n-1}) + \Vic_{n-1}. 
\end{align*}
For any of the above graphs we denote by $V_{n-1}^l$ with  $l\in\{\mathsf{LB}, \mathsf{RB}, \mathsf{RT}, \mathsf{LT}, \mathsf{M}\}$  their corresponding vertex sets.
For any five given recurrent sandpile configurations $\eta_\mathsf{LB}, \eta_\mathsf{RB}, \eta_\mathsf{RT}, \eta_\mathsf{LT}, \eta_\mathsf{M}$ on $\Vic_{n-1}$ we define for  $k\in\N$ the merged configuration  $\mu_k(\eta_\mathsf{LB}, \eta_\mathsf{RB}, \eta_\mathsf{RT}, \eta_\mathsf{LT}, \eta_\mathsf{M})$ on $\Vic_n$ as follows. Let 
$\phi_{n-1}:V_{n-1} \rightarrow V_{n-1}$ be the counterclockwise rotation of $V_{n-1}$ by 90 degrees and for $l\in\{\mathsf{LB}, \mathsf{RB}, \mathsf{RT}, \mathsf{LT}, \mathsf{M}\}$, let  $\tau_{l}:V_{n}^l\rightarrow V_{n-1}$ be the translation by the vector $c_l-c_\mathsf{LB}$, i.e.~moving $\Vic_{n-1}^l$ back into the position of $\Vic_{n-1}$ by subtracting $c_l-c_\mathsf{LB}$ from every vertex in $\Vic_{n-1}^l$. Define $\mu_k:\Vic_n\to\N$ as
\begin{align*}
        \mu_k(x) = \begin{cases}
            \eta_{\mathsf{LB}}(x) \ &\text{if} \ x \in V_{n-1}^\mathsf{LB} \backslash \{c_\mathsf{LB}\},\\
            \eta_\mathsf{RB}(\phi_{n-1}(\tau_\mathsf{RB}(x)) \ &\text{if} \ x\in V_{n-1}^\mathsf{RB}\backslash \{c_\mathsf{RB}\},\\
            \eta_\mathsf{RT}(\tau_\mathsf{LT}(x)) \ &\text{if} \ x\in V_{n-1}^\mathsf{RT}\backslash \{c_\mathsf{RT}\},\\
            \eta_\mathsf{LT}(\phi^3_{n-1}(\tau_\mathsf{LT}(x))) \ &\text{if} \ x\in V_{n-1}^\mathsf{LT}\backslash \{c_\mathsf{LT}\},\\
            \eta_\mathsf{M}(\tau_\mathsf{M}(x)) \ &\text{if} \ x\in V_{n-1}^\mathsf{M}\backslash \{c_\mathsf{LB}, c_\mathsf{RB}, c_\mathsf{RT}, c_\mathsf{LT}\},\\
            k + \eta_\mathsf{M}(\tau_\mathsf{M}(x)) \ &\text{if} \ x\in \{c_\mathsf{LB}, c_\mathsf{RB}, c_\mathsf{LT}\}, \\
            k + \eta_\mathsf{RT}(\tau_\mathsf{RT}(x)) \ &\text{if} \ x=c_\mathsf{RT},
        \end{cases}
\end{align*}
where $\phi^3_{n-1}$ is simply the threefold composition of the rotation $\phi_{n-1}$, i.e.~a rotation by 90 degrees clockwise.
See Figure \ref{fig:arrangement} for an illustration of $\mu_k(\eta_\mathsf{LB}, \eta_\mathsf{RB}, \eta_\mathsf{RT}, \eta_\mathsf{LT}, \eta_\mathsf{M})$. In view of the burning algorithm, the merged configuration for $k=3$ is also a recurrent sandpile. The next lemma gives a construction of the identity on $\Vic_n$ via merging the identities on five copies of $\Vic_{n-1}$. For an illustration of $\id_3$ see Figure \ref{fig:identity_3}.
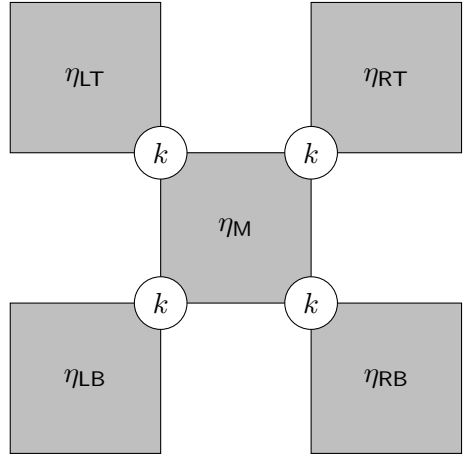
\begin{figure}[h]
    \centering
    \begin{tikzpicture}
        \draw[fill = lightgray] (0,0) rectangle ++(2,2) node[pos=.5] {$\eta_\mathsf{LB}$};
        \draw[fill = lightgray] (2,2) rectangle ++(2,2) node[pos=.5] {$\eta_\mathsf{M}$};
        \draw[fill = lightgray] (4,0) rectangle ++(2,2) node[pos=.5] {$\eta_\mathsf{RB}$};
        \draw[fill = lightgray] (4,4) rectangle ++(2,2) node[pos=.5] {$\eta_\mathsf{RT}$};
        \draw[fill = lightgray] (0,4) rectangle ++(2,2) node[pos=.5] {$\eta_\mathsf{LT}$};
        \node[shape=circle,draw=black, fill=white] (A) at (2,2) {$k$};
        \node[shape=circle,draw=black, fill=white] (B) at (4,2) {$k$};
        \node[shape=circle,draw=black, fill=white] (C) at (4,4) {$k$};
        \node[shape=circle,draw=black, fill=white] (D) at (2,4) {$k$};
    \end{tikzpicture}
    \caption{The arrangement of the merged sandpile configuration $\mu_k$.}
    \label{fig:arrangement}
\end{figure}
\begin{lem}\label{lem:id-construction}
If $\id_{n-1}$ is the identity element of the sandpile group $\mathcal{R}_{n-1}$ of $\Vic_{n-1}$, then the identity $\id_n$ of the sandpile group $\mathcal{R}_n$ of $\Vic_{n}$ can be constructed by merging five copies of $\id_{n-1}$ in the following manner:
    \begin{align*}
        \id_n = \mu_2(\id_{n-1},\id_{n-1},\id_{n-1},\id_{n-1},\id_{n-1}).
    \end{align*}
\end{lem}
\begin{proof}
The proof can be done by induction over the iteration $n\in\N$ with the additional induction statement that during the stabilization of the configuration $4 \id_n$ on $\Vic_n$, the number of particles accumulated at the sink $s_n$ is $2\mod 4$. The base case is straightforward since $\id_0 = 2 \mathds{1}_{\{s_0\}^\complement}$. 
 By Proposition \ref{prop:order-upper-bound}, for a recurrent configuration $\eta$ on $\Vic_n$ we have $(4 \eta )^\circ = \id_n$. For the inductive step we therefore stabilize $4 \mu_3(\id_{n-1},\id_{n-1},\id_{n-1},\id_{n-1},\id_{n-1})$, and notice that 
 \begin{align*}
 4 \mu_3(\id_{n-1},\id_{n-1},\id_{n-1},\id_{n-1},\id_{n-1}) & =  
 \mu_3(4\id_{n-1},4\id_{n-1},4\id_{n-1},4\id_{n-1},4\id_{n-1}) \\
 & + 9  (\delta_{c_\mathsf{LB}} + \delta_{c_\mathsf{RB}} + \delta_{c_\mathsf{RT}} + \delta_{c_\mathsf{LT}})
\end{align*} 
where the 9 additional particles at the cutpoints $c_\mathsf{LB}, c_\mathsf{RB}, c_\mathsf{RT}$ and $c_\mathsf{LT}$ are due to the fact that after each merging of configurations, 3 particles in each cutpoint are added. See the left hand side of Figure \ref{fig:first-step-identity} for an illustration of the merged configuration.
\begin{figure}
    \begin{align*}
    \begin{tikzpicture}
        \draw[fill = lightgray] (0,0) rectangle ++(2,2) node[pos=.5] {$4\id_{n-1}$};
        \draw[fill = lightgray] (2,2) rectangle ++(2,2) node[pos=.5] {$4\id_{n-1}$};
        \draw[fill = lightgray] (4,0) rectangle ++(2,2) node[pos=.5] {$4\id_{n-1}$};
        \draw[fill = lightgray] (4,4) rectangle ++(2,2) node[pos=.5] {$4\id_{n-1}$};
        \draw[fill = lightgray] (0,4) rectangle ++(2,2) node[pos=.5] {$4\id_{n-1}$};
        \node[shape=circle,draw=black, fill=white] (A) at (2,2) {12};
        \node[shape=circle,draw=black, fill=white] (B) at (4,2) {12};
        \node[shape=circle,draw=black, fill=white] (C) at (4,4) {12};
        \node[shape=circle,draw=black, fill=white] (D) at (2,4) {12};
    \end{tikzpicture} & &
        \begin{tikzpicture}
        \draw[fill = lightgray] (0,0) rectangle ++(2,2) node[pos=.5] {$\id_{n-1}$};
        \draw[fill = lightgray] (2,2) rectangle ++(2,2) node[pos=.5] {$4 \id_{n-1}$};
        \draw[fill = lightgray] (4,0) rectangle ++(2,2) node[pos=.5] {$\id_{n-1}$};
        \draw[fill = lightgray] (4,4) rectangle ++(2,2) node[pos=.5] {$4 \id_{n-1}$};
        \draw[fill = lightgray] (0,4) rectangle ++(2,2) node[pos=.5] {$\id_{n-1}$};
        \node[shape=circle,draw=black, fill=white] (A) at (2,2) {6};
        \node[shape=circle,draw=black, fill=white] (B) at (4,2) {6};
        \node[shape=circle,draw=black, fill=white] (C) at (4,4) {4};
        \node[shape=circle,draw=black, fill=white] (D) at (2,4) {6};        
    \end{tikzpicture}
    \end{align*}
 \caption{Stabilization of the subconfigurations of $4 \mu_3(\id_{n-1},\id_{n-1},\id_{n-1},\id_{n-1},\id_{n-1})$ in the three copies of $\Vic_{n-1}$ in the lower left and right as well as upper left. The configuration $4 \mu_3(\id_{n-1},\id_{n-1},\id_{n-1},\id_{n-1},\id_{n-1})$ is on the left and the result after stabilizing the subconfigurations on the right.}
 \label{fig:first-step-identity}
 \end{figure}
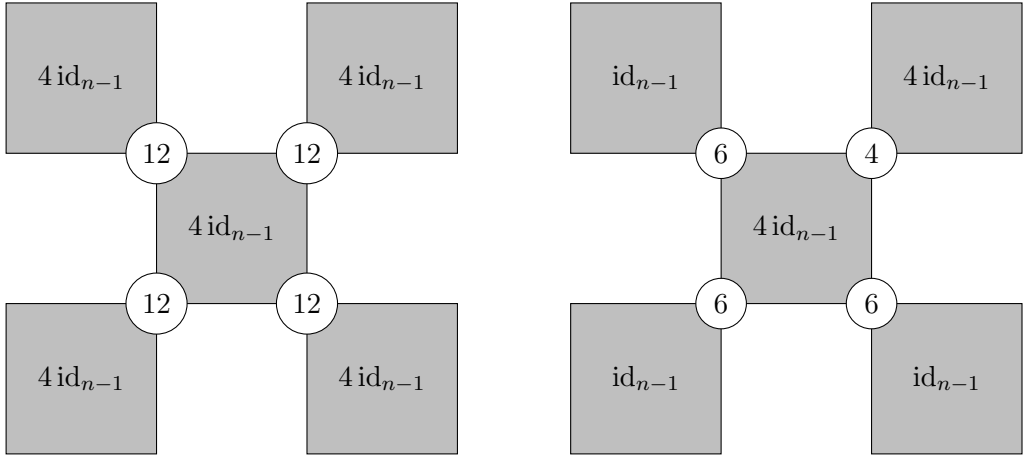
 By stabilizing $4 \id_{n-1}$ on the three subgraphs $\Vic_{n-1}^\mathsf{LB}$, $\Vic_{n-1}^\mathsf{RB}$ and $\Vic_{n-1}^\mathsf{LT}$, we first collect additional $2\mod 4$ particles in each of the three corresponding cutpoints $c_\mathsf{LB}, c_\mathsf{RB},c_\mathsf{LT}$ of the subgraphs respectively. Stabilizing sandpiles that are multiples of 4 at cutpoints, by Lemma \ref{lem:stability-outside-geodsubgraph} results in configuration $\id_{n-1}$ on the three subgraphs $\Vic_{n-1}^\mathsf{LB}$, $\Vic_{n-1}^\mathsf{RB}$, $\Vic_{n-1}^\mathsf{LT}$, and $6 = 12 + 2 \mod 4$ particles at the corresponding cutpoints. Figure \ref{fig:first-step-identity} shows the resulting configuration on the right hand side.
For stabilizing the central copy, we first recall that by Dhar's burning algorithm, we can consider the configuration in this copy as a usual configuration on $\Vic_{n-1}$ when removing (pinning) three particles at each of the cutpoint. With this in mind, for $\Vic_{n-1}^\mathsf{M}$ we have 
 \begin{align*}
        [4 \id_{n-1} + 3 (\delta_{c_\mathsf{LB}} + \delta_{c_\mathsf{RB}} + \delta_{c_\mathsf{LT}})] = [\id_{n-1}] - [\delta_{c_\mathsf{LB}} + \delta_{c_\mathsf{RB}} + \delta_{c_\mathsf{LT}}] = [\id_{n-1} - \delta_{c_\mathsf{LB}} - \delta_{c_\mathsf{RB}} - \delta_{c_\mathsf{LT}}],
    \end{align*}
and during stabilization $2 \mod 4$ particles are collected in the sink $(2\cdot 3^{n-1},2\cdot 3^{n-1})$. Recall that there are $3$ additional particles in the cutpoints $c_\mathsf{LT}$, $c_\mathsf{LB}$, and $c_\mathsf{RB}$  we pinned for stabilization, resulting in a total of 2 additional particles in $c_\mathsf{LT}$, $c_\mathsf{LB}$, and $c_\mathsf{RB}$ respectively, as illustrated on the left hand side of Figure \ref{fig:second-step-identity}.
Above we have used that $\id_{n-1}-\delta_{(0,0)}-\delta_{(3^{n-1},0)}-\delta_{(0,3^{n-1})}$ is recurrent on $\Vic_{n-1}$, since by induction $\id_{n-1}$ is constant two on the copies of $K_4$ on the lower left, lower right and upper left corner. Subtracting $1$ particle from the bottom left vertex from the constant two configuration on $K_4$, yields again a recurrent sandpile, as can be seen from Figure \ref{fig:recurrentK4}.
Finally using the same arguments as before, for the top right subgraph $\Vic_{n-1}^\mathsf{RT}$ we have 
    \begin{align*}
        [4 \id_{n-1} + 3\delta_{c_\mathsf{RT}}] = [\id_{n-1}] - [\delta_{c_\mathsf{RT}}],
    \end{align*}
    resulting in the final configuration $\mu_2(\id_{n-1},\id_{n-1},\id_{n-1},\id_{n-1},\id_{n-1})$ displayed in Figure \ref{fig:second-step-identity} and during stabilization $2 \mod 4$ particles are collected in the sink, and this concludes the inductive step and thus the proof of the claim.
\end{proof}
 \begin{figure}
        \begin{align*}
    \begin{tikzpicture}
        \draw[fill = lightgray] (0,0) rectangle ++(2,2) node[pos=.5] {$\id_{n-1}$};
        \draw[fill = lightgray] (2,2) rectangle ++(2,2) node[pos=.5] {$\id_{n-1}$};
        \draw[fill = lightgray] (4,0) rectangle ++(2,2) node[pos=.5] {$\id_{n-1}$};
        \draw[fill = lightgray] (4,4) rectangle ++(2,2) node[pos=.5] {$4 \id_{n-1}$};
        \draw[fill = lightgray] (0,4) rectangle ++(2,2) node[pos=.5] {$\id_{n-1}$};
        \node[shape=circle,draw=black, fill=white] (A) at (2,2) {2};
        \node[shape=circle,draw=black, fill=white] (B) at (4,2) {2};
        \node[shape=circle,draw=black, fill=white] (C) at (4,4) {6};
        \node[shape=circle,draw=black, fill=white] (D) at (2,4) {2};        
    \end{tikzpicture} & &
        \begin{tikzpicture}
        \draw[fill = lightgray] (0,0) rectangle ++(2,2) node[pos=.5] {$\id_{n-1}$};
        \draw[fill = lightgray] (2,2) rectangle ++(2,2) node[pos=.5] {$\id_{n-1}$};
        \draw[fill = lightgray] (4,0) rectangle ++(2,2) node[pos=.5] {$\id_{n-1}$};
        \draw[fill = lightgray] (4,4) rectangle ++(2,2) node[pos=.5] {$\id_{n-1}$};
        \draw[fill = lightgray] (0,4) rectangle ++(2,2) node[pos=.5] {$\id_{n-1}$};
        \node[shape=circle,draw=black, fill=white] (A) at (2,2) {2};
        \node[shape=circle,draw=black, fill=white] (B) at (4,2) {2};
        \node[shape=circle,draw=black, fill=white] (C) at (4,4) {2};
        \node[shape=circle,draw=black, fill=white] (D) at (2,4) {2};        
    \end{tikzpicture}
    \end{align*}
\caption{The configuration on the left shows $\mu_3(\id_{n-1},\id_{n-1},\id_{n-1},\id_{n-1},\id_{n-1})$ after it has been stabilized in the lower left, lower right, middle and upper left copy of $\Vic_{n-1}$, after which $2$ particles have been collected at the cutpoints $c_\mathsf{RB}, c_\mathsf{LB},c_\mathsf{LT}$. The configuration on the right shows the final stable configuration, after carrying out the stabilization in the upper right copy of $\Vic_{n-1}$.}
\label{fig:second-step-identity}
\end{figure}
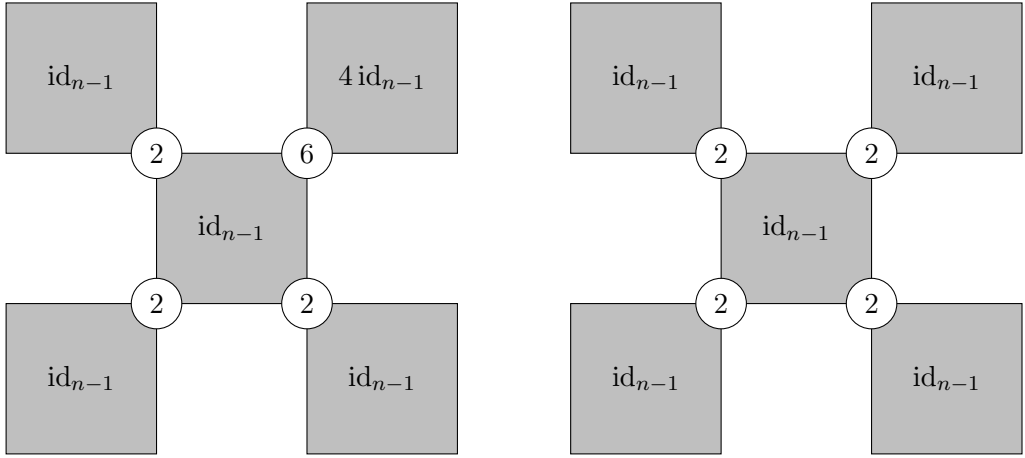

\begin{figure}[h]
    \centering
    \begin{tikzpicture}[scale=0.5, transform shape]
    \input{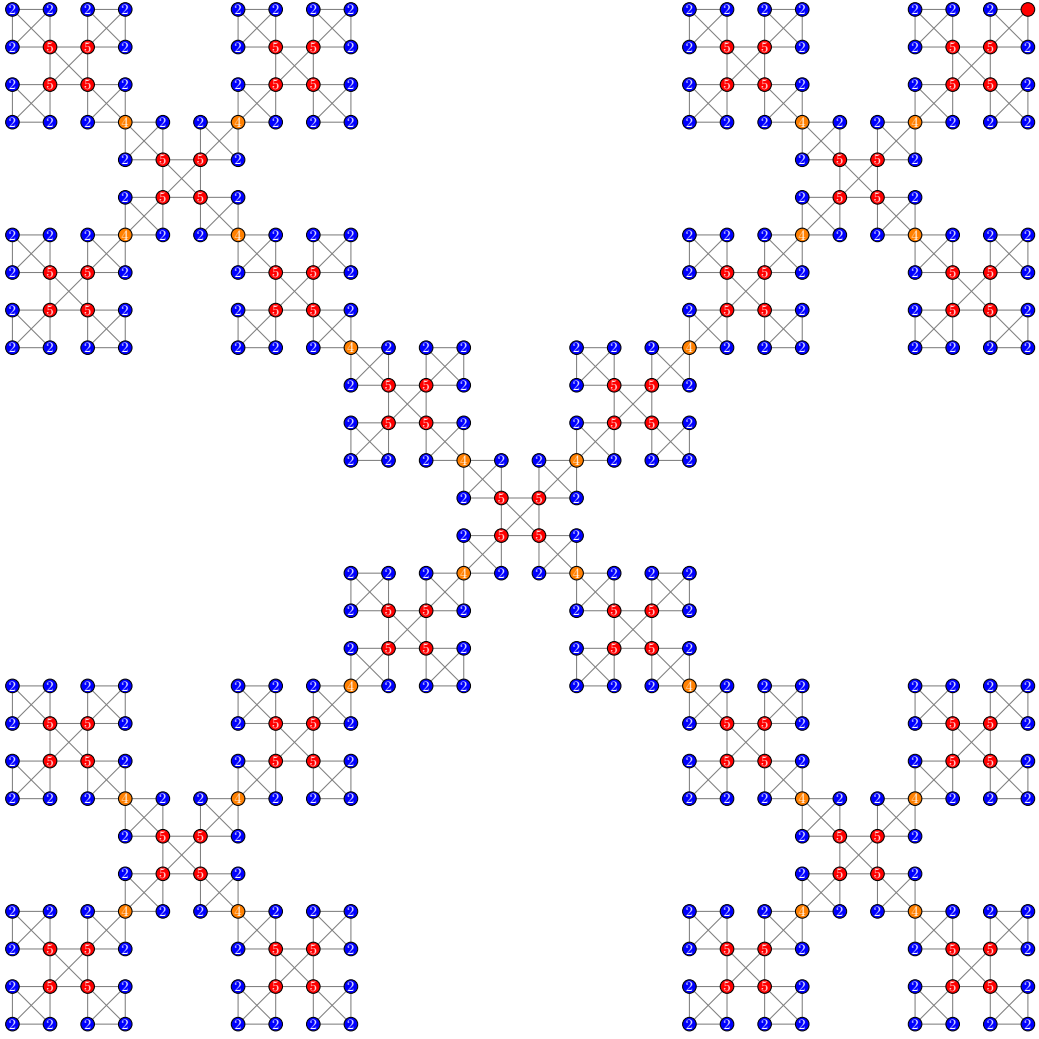}
    \end{tikzpicture}
    \caption{The identity $\id_3$ on $\Vic_3$ with colors according to the configuration height (blue = 2, orange = 4, red = 5).}
    \label{fig:identity_3}
\end{figure}

\subsection{Sandpile group of the Vicsek graph $\Vic$ and stabilisation in infinite volume}\label{sec:sandpile-group}

For any finite Abelian group $(G,+)$ with identity $0$, we define the set of elements of order $k$ by
\begin{align*}
    \orderset(G,k) = \{g\in G: g^k = 0\}.
\end{align*}
We show first an upper bound for the elements of order 2, which then applies to the case of finite Vicsek graphs $\Vic_n$, $n\in\N$.

\begin{lem}\label{lem:preimage-count}
    Let $(G,+)$ be a finite Abelian group where every element has order either $2$ or $4$, and take any $g\in G$. Define the mapping $f:G\rightarrow G$ by $f(h)=h+h$. Then we have
    \begin{align*}
        |f^{-1}(\{g\})|\in\{0,|\orderset(G,2)|\}.
    \end{align*}
\end{lem}
\begin{proof}
Assume that $f^{-1}(\{g\})\neq\emptyset$, and choose an element from $f^{-1}(\{g\})$ and denote it by $g/2$. For every $h\in\orderset(G,2)$ we have that $h+g/2\in f^{-1}(\{g\})$ since
\begin{align*}
f(h+g/2)=h+h+g/2+g/2=g,
\end{align*}
thus $|f^{-1}(\{g\})|\leq |\orderset(G,2)|$. On the other hand for every $k\in f^{-1}(\{g\})$, we have $k-g/2\in\orderset(G,2)$ in view of
    \begin{align*}
        k-g/2+k-g/2=2k-g=0.
    \end{align*}
Thus it holds $|f^{-1}(\{g\})|=|\orderset(G,2)|$ and this proves the claim.
\end{proof}
Applying now this lemma to the sandpile group $\mathcal{R}_n$ of $\Vic_n$ yields the following.

\begin{lem}\label{lem:order-2-upperbound}
For every $n\in\N$, it holds 
\begin{align*}
|\orderset(\mathcal{R}_{n},2)| \leq |\orderset(\mathcal{R}_{n-1},2)|^5.
\end{align*}
\end{lem}
\begin{proof}
Let $\eta \in \orderset(\mathcal{R}_{n},2)$ and denote the restrictions of $\eta$ to the five  copies of $\Vic_{n-1}$ in $\Vic_n$ by
    \begin{align*}
        \eta_\mathsf{LB} & = \eta|_{V_{n-1}^\mathsf{LB}}, \quad \eta_\mathsf{RB} = \eta|_{V_{n-1}^\mathsf{RB}},\\
        \eta_\mathsf{LT} & = \eta|_{V_{n-1}^\mathsf{LT}}, \quad \eta_\mathsf{RT}  = \eta|_{V_{n-1}^\mathsf{RT}}, \\ 
        \eta_\mathsf{M}  &= \eta|_{V_{n-1}^\mathsf{M}}.
    \end{align*}
It holds that $\eta_\mathsf{LB}, \eta_\mathsf{RB}, \eta_\mathsf{LT}\in \orderset(\mathcal{R}_{n-1},2)$ because after stabilization of $2\eta$, due to the Dhar's burning bijection, the configurations on the respective restrictions $\Vic_{n-1}^\mathsf{LB},\Vic_{n-1}^\mathsf{RB},\Vic_{n-1}^\mathsf{LT}$ do not change any further and have to match $\id_{n-1}$ by Lemma \ref{lem:id-construction}.
We consider now $\eta$ restricted on the central copy, that is $ \eta_\mathsf{M}$,  and denote by $s(\eta_i)$ the number of particles sent into the sink vertex $c_i$  when stabilizing $2\eta_i$ for $i\in\{\mathsf{LB},\mathsf{RB},\mathsf{LT}\}$. Then for $\eta$ to have order $2$ in $\mathcal{R}_n$, we must have
    \begin{align*}
        (2\eta_\mathsf{M}+(3+s(\eta_\mathsf{LB}))\delta_{c_\mathsf{LB}}+(3+s(\eta_\mathsf{RB}))\delta_{c_\mathsf{RB}}+(3+s(\eta_\mathsf{LT}))\delta_{c_\mathsf{LT}})^\circ=\mathsf{id}_{n-1},
    \end{align*}
equation which rewritten in terms of the sandpile group gives
\begin{align}
[2\eta_\mathsf{M}]=-[(3+s(\eta_\mathsf{LB}))\delta_{c_\mathsf{LB}}+(3+s(\eta_\mathsf{RB}))\delta_{c_\mathsf{RB}}+(3+s(\eta_\mathsf{LT}))\delta_{c_\mathsf{LT}}].\label{eq:order2-for-mid-configuration}
\end{align}
By Lemma \ref{lem:preimage-count} there are at most $|\orderset(\mathcal{R}_{n-1},2)|$ possible choices for $\eta_\mathsf{M}$.
If we write $t(\eta)$ for the number of particles sent into the sink $c_{\mathsf{RT}}$ when stabilizing $2\eta$ on $\Vic_{n}\backslash \Vic_{n-1}^\mathsf{RT}$, that is all but the top right copy of $\Vic_{n-1}$ in $\Vic_n$. So in order for $\eta$ on $\Vic_n$ to be of order $2$ we must have
\begin{align*}
(2\eta_\mathsf{RT}+(3+t(\eta))\delta_{c_\mathsf{RT}})^\circ=\mathsf{id}_{n-1},
 \end{align*}
 or in terms of the sandpile group
    \begin{align*}
        [2\eta_\mathsf{RT}]=-[(3+t(\eta))\delta_{c_\mathsf{RT}}].
    \end{align*}
Once again by Lemma \ref{lem:preimage-count} we have at most $|\orderset(\mathcal{R}_{n-1},2)|$ possible choices for $\eta_\mathsf{RT}$. By exactly the same argument, for every $i\in\{\mathsf{LB},\mathsf{RB},\mathsf{RT},\mathsf{LT},\mathsf{M}\}$ there are at most $|\orderset(\mathcal{R}_{n-1},2)|$ choices for $\eta_i$, leading to a total of $|\orderset(\mathcal{R}_{n-1},2)|^5$ choices for $\eta$.
\end{proof}

We are now ready to prove Theorem \ref{thm:sandpile-group}.
\begin{proof}[Proof of Theorem \ref{thm:sandpile-group}]
For any $n\in\N$ and every $[\eta]\in\mathcal{R}_{n}$, we have by Proposition \ref{prop:order-upper-bound} that $        \order_{\mathcal{R}_n}([\eta])\leq 4$, which together with the fundamental theorem for finite, Abelian groups implies the existence of $k,l\in\N$ such that
    \begin{align}\label{eq:cong-of-sp}
        \mathcal{R}_{n} \cong \Z_2^l\times\Z_4^k,\quad \text{with} \quad 2^l\cdot 4^k=16^{5^n}.
    \end{align}
 Thus there exist $\sum_{i=0}^{l+k}{l+k \choose i}=2^{l+k}$
    elements $[\eta]\in\mathcal{R}_{n}$ such that $[\eta]+[\eta]=[\id_n]$. Since $|\orderset(\mathcal{R}_{n},2)|\leq|\orderset(\mathcal{R}_{n-1},2)|^5$ and $|\orderset(\mathcal{R}_{0},2)|=4$ by Lemma \ref{lem:order-2-upperbound}, we obtain
 $|\orderset(\mathcal{R}_{n},2)|=2^{2\cdot 5^n}$, because this is the minimal number of order 2 elements of all possible choices of $l$ and $k$. Under the constraints on $l$ and $k$ from (\ref{eq:cong-of-sp}), it then follows that
$l=0$ and $k=2\cdot 5^n$
which together with Equation \eqref{eq:cong-of-sp} completes the proof.
\end{proof}
By considering the order of the elements of the sandpile groups $\mathcal{R}_n$ of $\Vic_n$ for all $n\in\N$, we can now 
give an alternative proof for non-stabilization on the Vicsek graph, which is weaker than Theorem \ref{thm:vicsek-explosion}, since we cannot compute the explicit stabilisation probability.

\begin{proof}[Second Proof of Non-Stabilization on the Vicsek graph]
From Theorem \ref{thm:sandpile-group}, for every $n\in\N$ and every $[\eta]\in\mathcal{R}_{n}$ we have
$\order_{\mathcal{R}_{n}}([\eta])\leq 4$ which together with Theorem \ref{thm:relation-order-stabilziation} implies the existence of a constant $c>0$ such that for all $R>0$ and all $x\in V$ we have
$\mu(\mathsf{diam}(\mathsf{T}(\eta+\delta_x))>R)>c$. Here $\mu$ is the infinite volume limit on $\Vic$ and $\eta$ is a sandpile sampled from $\mu$. Thus we have
    \begin{align*}
        \mu(\eta+\delta_o\text{ stabilizes})=\lim_{R\rightarrow\infty}\mu(\mathsf{diam}(\mathsf{T}(\eta+\delta_x))\leq R)<1-c<1,
\end{align*}
which proves the explosion with probability at least $c>0$. 
\end{proof}

\textbf{Acknowledgments.} The research of R.~Kaiser and E.~Sava-Huss is supported by the Austrian Science Fund (FWF) P 34129. N.~Heizmann expresses his gratitude to the University of Innsbruck for hospitality and support. His visit to Innsbruck was financially supported by the same research grant (FWF) P 34129.  We are very grateful to the two anonymous referees for their suggestions and positive criticism that improved substantially the quality and the presentation of the paper.

\bibliographystyle{alpha}
\bibliography{lit}

\textsc{Nico Heizmann}, Department of Mathematics, Chemnitz University of Technology, Germany\\
\texttt{nico.heizmann@math.tu-chemnitz.de}

\textsc{Robin Kaiser}, Institut für Mathematik, Universität Innsbruck, Austria.\\
\texttt{Robin.Kaiser@uibk.ac.at}

\textsc{Ecaterina Sava-Huss}, Institut für Mathematik, Universität Innsbruck, Austria.\\
\texttt{Ecaterina.Sava-Huss@uibk.ac.at}
\end{document}